\documentclass[reqno,twoside,12pt]{amsart}
\hoffset - 1.5cm

\usepackage{amsmath}
\usepackage{amssymb}
\usepackage{graphicx}
\usepackage{amstext}
\usepackage{amsopn}
\usepackage{amsfonts}
\usepackage{xcolor}
\usepackage{setspace}
\usepackage{enumerate}
\usepackage[polish,english]{babel}
\usepackage[utf8]{inputenc}
\usepackage{polski}
\usepackage[all]{xypic}

\newtheorem{Remark}{Remark}[section]
\newtheorem{Corollary}[Remark]{Corollary}
\newtheorem{Definition}[Remark]{Definition}
\newtheorem{Example}[Remark]{Example}
\newtheorem{Fact}[Remark]{Fact}
\newtheorem{Lemma}[Remark]{Lemma}
\newtheorem{Proposition}[Remark]{Proposition}
\newtheorem{Theorem}[Remark]{Theorem}

\newcommand{\bH}{\mathbb{H}}
\newcommand{\bI}{\mathbb{I}}

\newcommand{\bN}{\mathbb{N}}
\newcommand{\bO}{\mathbb{O}}

\newcommand{\bR}{\mathbb{R}}

\newcommand{\bU}{\mathbb{U}}
\newcommand{\bW}{\mathbb{W}}
\newcommand{\bV}{\mathbb{V}}
\newcommand{\bX}{\mathbb{X}}

\newcommand{\bZ}{\mathbb{Z}}

\newcommand{\cC}{\mathcal{C}}

\newcommand{\cG}{\mathcal{G}}
\renewcommand{\cH}{\mathcal{H}}

\newcommand{\cK}{\mathcal{K}}

\newcommand{\cN}{\mathcal{N}}

\newcommand{\cT}{\mathcal{T}}
\newcommand{\cU}{\mathcal{U}}

\newcommand{\cW}{\mathcal{W}}

\numberwithin{equation}{section} \errorcontextlines=0
\newcommand{\degso}{\nabla_{\sone}\textrm{-}\mathrm{deg}}
\newcommand{\degg}{\nabla_{G}\textrm{-}\mathrm{deg}}
\newcommand{\degh}{\nabla_{H}\textrm{-}\mathrm{deg}}

\newcommand{\diag}{\mathrm{ diag \;}}

\newcommand{\sign}{\mathrm{ sign \;}}

\newcommand{\sone}{SO(2)}

\newcommand{\h}{\mathbb{H}}
\newcommand{\sub}{\overline{\operatorname{sub}}}
\newcommand{\G}{\mathcal{G}}

\newcommand{\BIF}{\mathcal{BIF}}

\begin{document}

\title[Global bifurcation index]{Global bifurcation index of critical orbit of strongly indefinite functional}%

\author{Anna Go\l\c{e}biewska}
\address{Faculty of Mathematics and Computer Science\\
Nicolaus Copernicus University \\
PL-87-100 Toru\'{n} \\ ul. Chopina $12 \slash 18$ \\
Poland, ORCID 0000--0002--2417--9960}

\email{Anna.Golebiewska@mat.umk.pl}

\author{Piotr Stefaniak}
\address{School of Mathematics,
West Pomeranian University of Technology \\PL-70-310 Szcze\-cin, al. Piast\'{o}w $48\slash 49$, Poland, ORCID 0000--0002--6117--2573}

\email{pstefaniak@zut.edu.pl}

\numberwithin{equation}{section}
\allowdisplaybreaks
\date{\today}

\maketitle

\begin{abstract}
In this paper we study an index of a critical orbit, defined in terms of the degree for invariant strongly indefinite functionals. We establish a relationship of this index with the index of a critical point of the mapping restricted to the space normal to the orbit. The second aim of the article is to use the index of a critical orbit to prove the bifurcation of nontrivial solutions of nonlinear elliptic systems with Neumann boundary conditions. We consider also the existence of unbounded sets of solutions of such systems.
\end{abstract}

\section{Introduction}
The aim of this paper is to investigate an index of a critical orbit of a strongly indefinite functional. The problem of studying critical orbits and, in general, critical points, arises in the variational method for differential equations. With such an equation one can associate a functional defined on a suitable Hilbert space. Solutions of the equation are in one to one correspondence with critical points of the functional. That is why studying the existence of critical points and the structure of the set of such points is very important. Such problems have been widely investigated.

To detect critical points one can define an index of a point, for example in terms of the Conley index or of the degree. Nontriviality of the index implies the existence of the solutions.  However, this index is usually considered with the assumption that the critical points are isolated.

On the other hand, a differential equation, and therefore also the associated functional,  possesses often some additional symmetries, given for example by the symmetry of the domain or of the potential. In such a case, the assumption that critical points are isolated, does not have to be satisfied. More precisely, one should investigate a critical orbit instead of a critical point.

The index of a critical orbit, given in terms of the Conley index, has been recently defined in \cite{PRS}. This index can be used to prove the existence of critical orbits as well as the bifurcation of orbits of nontrivial solutions from the set of trivial ones. Moreover, using the relationship between the Conley index and the degree for strongly indefinite functionals given in \cite{GolRyb2013} one can prove, under additional assumptions, the existence of connected sets of orbits of nontrivial solutions bifurcating from the set of trivial ones.

However, in the general case, the Conley index cannot be used to describe the structure of connected sets of critical points. To study such sets, one can use for example the degree theory and investigate the phenomenon of a global bifurcation and a version of the famous Rabinowitz alternative. In particular, one can prove the existence of unbounded sets of solutions.

Therefore, the first aim of our paper is to define the index of an orbit as a degree on a neighbourhood of this orbit.  A sufficiently small neighbourhood can be described with the use of twisted spaces. The theory of such spaces plays an important role in the proofs of the properties of the mapping defined on this neighbourhood. In particular, it allows to describe the isotropy groups of elements of the critical orbit. As a consequence of these properties, we obtain the formula for the degree, see Theorem \ref{thm:formuladeg}. In some cases, namely for the so called admissible pairs of group, this result allows to reduce comparing the degrees of critical orbits to comparing the degrees of isolated critical points. We prove this in the finite dimensional case in Corollary \ref{cor:differentfinitedegrees} and in the infinite dimensional one in Theorem \ref{cor:difdeg}.
Next we apply these results to study the phenomenon of global bifurcation from a critical orbit. In particular we formulate a symmetric version of the Rabinowitz alternative, see Theorem \ref{thm:Rab}.

The second aim of our paper is the application of abstract results to the systems of differential equations. Namely, we consider
\begin{equation}\label{eq:neumannwstep}
\left\{
\begin{array}{rclcl}   A \triangle u & =& \nabla_u F(u,\lambda )   & \text{ in   } & \cU \\
                   \frac{\partial u}{\partial \nu} & =  & 0 & \text{ on    } & \partial\cU,
\end{array}\right.
\end{equation}
for $\cU$ being an open, bounded and $SO(N)$-invariant set, $A=\diag(\pm1, \ldots, \pm1)$ and $F$ satisfying some additional assumptions. Primarily, we consider the problem with $\Gamma$-symmetric potential $F$ for $\Gamma$ being a compact Lie group. For this system we prove the global bifurcation of nontrivial solutions from the set of trivial ones, see Theorem \ref{thm:glob}. Since the problem is symmetric, the trivial solutions do not have to be isolated. That is why we consider bifurcations from an orbit.

In the paper \cite{GolKluSte} we have considered the system \eqref{eq:neumannwstep} with $\nabla_uf(u, \lambda)=\lambda \nabla F(u)$, $\Omega=B^N$ and $a_i=1$. We have obtained global bifurcations of nontrivial solutions from an orbit of trivial ones. To this end we have examined the change of the Conley index and applied its relationship with the degree theory. It is worth pointing out that using this method we are able to find connected sets of solutions, but in this way we cannot study the structure of such sets, in particular we cannot prove the existence of unbounded sets of solutions.
Such information can be obtained via the degree theory.

Our last aim is to discuss the existence of unbounded continua of solutions, see Theorem \ref{thm:unbounded}.
To this end we study when the latter possibility in the Rabinowitz alternative can be excluded.

A similar method has been used by the second author in the papers \cite{RybShiSte} and \cite{RybSte}. In these papers some elliptic systems on spheres and on geodesic balls have been studied. Excluding one of the possibilities in the Rabinowitz alternative, it has been shown that in these systems all the continua of nontrivial solutions, bifurcating from the set of trivial ones, are unbounded.
We refer to these papers for further references and discussion of other results concerning unbounded continua of solutions.
In \cite{GolKlu} the first author has studied the existence of such sets with the use of the bifurcation index at the infinity, being an element of the Euler ring.
Nontriviality of this index implies the existence of such continua.
We emphasise that all these results have been obtained under the assumption that the set of trivial solutions is of the form $\{0\}\times\bR$.

\section{Preliminaries}

Throughout this section $G$ denotes a compact Lie group.

\subsection{The Euler ring}

Let $U(G)$ be the Euler ring of the group $G$. The definition and the properties of this ring can be found in  \cite{TomDieck1}, \cite{TomDieck}. Below we recall only some facts which are useful in the rest of the paper.

Denote by $\chi_G(X) \in U(G)$ the $G$-equivariant Euler characteristic of a pointed  finite $G$-CW-complex $X$. Recall that the actions in $U(G)$ are defined by
\begin{equation}\label{eq:actionsUG}
\left.\begin{array}{rcl}
\chi_G(X)+ \chi_G(Y)&=&\chi_G(X\vee Y),\\
\chi_G(X)\star \chi_G(Y)&=&\chi_G(X\wedge Y),
\end{array} \right.
\end{equation}
where $X\vee Y$ is the wedge sum and $X\wedge Y$ is the smash product of pointed finite $G$-CW-complexes  $X,Y$.

If $X$ is a $G$-CW-complex without a base point, then by $X^+$ we denote a pointed $G$-CW-complex $X\cup\{\ast\}$.

\begin{Lemma}
$(U(G),+,\star)$, with the actions given by \eqref{eq:actionsUG}, is a commutative ring with unit $\mathbb{I}=\chi_G(G/G^+)$.
\end{Lemma}

Denote by $\sub[G]$ the set of conjugacy classes $(H)_G$ of closed subgroups of the group $G$.

\begin{Lemma}
$(U(G),+)$ is a free abelian group with the basis $\chi_G(G/H^+),\ (H)_G\in\sub[G]$.
\end{Lemma}

\begin{Remark}
From the above lemma it follows that one can identify the Euler ring $U(G)$ with the $\bZ$-module generated by the conjugacy classes of the closed subgroups of $G$, i.e. with $\bigoplus_{(H)_G\in\sub[G]} \bZ$, see Corollary IV.1.9. of \cite{TomDieck}. Therefore we can index the coordinates of elements of $U(G)$ by the classes $(H)_G\in\sub[G]$.
\end{Remark}

\subsection{$G$-Morse functions}
In this subsection we recall the notion of invariant Morse functions, which has been introduced by Mayer in \cite{Mayer}.
Let $\bV$ be a finite dimensional, orthogonal representation of the group $G$ and $\Omega\subset\bV$ an open, bounded, $G$-invariant subset.

Fix a $G$-invariant map $\varphi\in C^2(\bV,\bR)$ and $v_0\in (\nabla\varphi)^{-1}(0)$. Denote by $T_{v_0} G(v_0)$ the tangent space to the orbit $G(v_0)$ at $v_0$ and let $\bW=(T_{v_0} G(v_0))^{\bot}$. Put $H=G_{v_0}$ and let $\bW^H$ denote the set of fixed points of the action of $H$ on $\bW$.

\begin{Lemma}\label{lem:hessian}
Under the above notation,
\begin{equation*}
\left.\begin{array}{rccc}
& T_{v_0} G(v_0)& &T_{v_0} G(v_0) \\
&\oplus& &\oplus\\
\nabla^2 \varphi(v_0)\colon &\bW^H&\to&\bW^H\\
&\oplus& &\oplus\\
&(\bW^H)^{\bot}& &(\bW^H)^{\bot}
\end{array}\right.
\end{equation*}
has the following form
\begin{equation*}
\nabla^2 \varphi(v_0)=\left[\begin{array}{ccc}
0&0&0\\
0&B(v_0)&0\\
0&0&C(v_0)
\end{array}\right].
\end{equation*}
\end{Lemma}
The proof of this lemma can be found in \cite{Geba}.

A $G$-invariant function $\varphi\in C^1(\bV,\bR)$ is called $\Omega$-admissible if $(\nabla\varphi)^{-1}(0)\cap\partial\Omega=\emptyset$. A $G$-invariant map $h\in C^1(\bV\times[0,1],\bR)$ is an $\Omega$-admissible homotopy if $(\nabla_v h)^{-1}(0)\cap(\partial\Omega\times[0,1])=\emptyset$.

We say that $\Omega$-admissible $G$-invariant functions $\varphi_1,\varphi_2\in C^1(\bV,\bR)$ are  $\Omega$-homotopic if there exists an $\Omega$-admissible homotopy $h\in C^1(\bV\times[0,1],\bR)$ such that $\nabla_v h (\cdot,0)=\nabla \varphi_1$ and $\nabla_v h (\cdot,1)=\nabla \varphi_2$.

We call an $\Omega$-admissible $G$-invariant function $\varphi\in C^2(\bV,\bR)$ a $G$-invariant $\Omega$-Morse function if for every $v_i\in(\nabla\varphi)^{-1}(0)\cap \Omega$ the orbit $G(v_i)$ is a non-degenerate critical orbit, i.e. $\dim \ker \nabla^2 \varphi(v_i)=\dim G(v_i)$.

Let $\varphi\in C^2(\bV,\bR)$ be a $G$-invariant $\Omega$-Morse function. We say that $\varphi$ is a special $G$-invariant $\Omega$-Morse function if for every $v_i\in(\nabla\varphi)^{-1}(0)\cap \Omega$  the orbit $G(v_i)$ is a special non-degenerate critical orbit, i.e. $m^-(\nabla^2 \varphi(v_i))=m^-(B(v_i))$, where $m^-(\cdot)$ is the Morse index and $B(v_i)$ is defined in Lemma \ref{lem:hessian}.

\begin{Lemma}\label{lem:special}
Let $\phi\in C^2(\bV,\bR)$ be a $G$-invariant $\Omega$-Morse function. Then there exists an open, bounded, $G$-invariant subset $\Omega_0\subset cl(\Omega_0)\subset \Omega$ such that $(\nabla\phi)^{-1}(0)\cap \Omega\subset \Omega_0$ and a special $G$-invariant $\Omega$-Morse function $\tilde{\phi}\colon\bV\to\bR$ such that
\begin{enumerate}
\item $(\nabla\tilde{\phi})^{-1}(0)\cap\Omega\subset \Omega_0$,
\item $\tilde{\phi}=\phi$ on $\bV\setminus\Omega_0$,
\item $(\nabla\phi)^{-1}(0)\subset (\nabla\tilde{\phi})^{-1}(0)$.
\end{enumerate}
\end{Lemma}

For the proof of the above lemma we refer to \cite{Mayer}.

\begin{Remark}\label{rem:special}
We say that a special $G$-invariant $\Omega$-Morse function $\tilde{\phi}$ satisfying conditions (1)--(3) of Lemma \ref{lem:special} is associated with $\phi$.
Note that the condition (2) implies that any two functions $\tilde{\phi}_1, \tilde{\phi}_2$ associated with  $\phi$ are $\Omega$-homotopic.
\end{Remark}

\subsection{Twisted spaces}\label{subsec:twisted}

For a closed subgroup $H$ of $G$ and a given $H$-space $X$ one can construct a $G$-space associated with $X$. Consider the  product $G\times X$ with an $H$-action given by $(h,(g,x))\mapsto(gh^{-1}, hx)$. Denote by $G\times_H X$ the space of orbits of this action and note that $G\times_H X$ is a $G$-space with the $G$-action given by $(g',[g,x])\mapsto [g'g,x]$ for $g'\in G$ and $[g,x]\in G\times_H X$. We call this space the twisted product over $H$. The properties of the twisted product can be found for example in \cite{Kawakubo}, \cite{TomDieck}.

Let $\bV$ be a finite dimensional, orthogonal representation of $G$. Fix $v_0\in \bV$, put $H=G_{v_0}$ and as previously put $\bW=(T_{v_0} G(v_0))^{\bot}$. It is easy to prove that $\bW$ is an $H$-space.
Denote by $B_{\epsilon}(v_0,\bW)$ the open ball of radius $\epsilon$ centred at $v_0$.

For the proofs of the two next theorems see for example \cite{Bredon}, \cite{DuisKolk}, \cite{Mayer}.

\begin{Theorem}[Slice theorem]\label{thm:slice}
There exists $\epsilon>0$ such that the mapping $G\times_H \bW \to \bV$ defined by $[g,w]\mapsto gw$ induces a $G$-equivariant diffeomorphism $\theta$ from $G\times_H B_{\epsilon}(v_0,\bW)$ to an open $G$-invariant neighbourhood  $G\cdot B_{\epsilon}(v_0,\bW)=\{g w\colon g\in G, w\in B_{\epsilon}(v_0,\bW) \}$ of the orbit $G(v_0)$.
\end{Theorem}

\begin{Theorem}\label{thm:slice2}
If the $\epsilon$ is given by Theorem \ref{thm:slice} and $g\cdot B_{\epsilon}(v_0,\bW)\cap B_{\epsilon}(v_0,\bW) \neq \emptyset$, then $g\in G_{v_0}$.
\end{Theorem}

\section{Degree of a critical orbit}

 In this section we recall some facts concerning the degree for equivariant gradient maps and the degree for invariant strongly indefinite functionals. Next we prove the results allowing to simplify the computation of the degree at some neighbourhood of a critical orbit.

\subsection{The definition of the degree }\label{subsec:defdegree}

Let $G$ be a compact Lie group and $\bV$ be a finite dimensional orthogonal $G$-representation. For a $G$-invariant function $\varphi \in C^1(\bV, \bR)$ and an open, bounded, $G$-invariant set $\Omega \subset \bV$ such that $\partial \Omega \cap (\nabla \varphi)^{-1}(0) = \emptyset,$ Gęba has defined in \cite{Geba} the degree $\degg(\nabla \varphi, \Omega)$, being an element of the Euler ring $U(G)$. In this subsection we recall some elements of this definition. Namely, since in our paper we use only $G$-invariant $\Omega$-Morse functions, we recall the definition of the degree for such mappings.

We start with the definition of the degree for a special $G$-invariant $\Omega$-Morse function $\varphi\in C^2(\bV,\bR)$.
In this case $(\nabla\varphi)^{-1}(0)\cap cl(\Omega)=G(v_1)\cup\ldots\cup G(v_k)$ and $G(v_i)\cap G(v_j)=\emptyset$ for $i\neq j$. Set $\mathcal{V}= \{v_1,\ldots, v_k\}$. For $(K)_G\in\sub[G]$ put
\begin{equation}\label{eq:stopienspec}
\nabla_G\text{-}\deg_{(K)}\left(\nabla\varphi,\Omega\right)=\sum\limits_{v\in\mathcal{V}, (G_v)_G=(K)_G} (-1)^{m^-\left(\nabla^2\varphi(v)\right)} \in \bZ
\end{equation}
and define the degree $\nabla_G\text{-}\deg(\nabla\varphi,\Omega)\in U(G)$ by the formula
\begin{equation*}
\nabla_G\text{-}\deg\left(\nabla\varphi,\Omega\right)=\sum\limits_{(K)_G \in \sub[G]}\nabla_G\text{-}\deg_{(K)}\left(\nabla\varphi,\Omega\right)\cdot \chi_G\left(G/K^+\right).
\end{equation*}

\begin{Theorem}\label{thm:homotopic}
Let $\varphi_1,\varphi_2\colon\bV\to\bR$ be special $G$-invariant $\Omega$-Morse functions. If $\varphi_1,\varphi_2$ are $\Omega$-homotopic then $\nabla_G\text{-}\deg(\nabla\varphi_1,\Omega)=\nabla_G\text{-}\deg(\nabla\varphi_2,\Omega)$.
\end{Theorem}

Let $\phi\in C^2(\bV,\bR)$ be a $G$-invariant $\Omega$-Morse function and $\tilde{\phi}\in C^2(\bV,\bR)$ a special $G$-invariant $\Omega$-Morse function associated with $\phi$.
Define
\begin{equation*}
\nabla_G\text{-}\deg\left(\nabla\phi,\Omega\right)=\nabla_G\text{-}\deg\left(\nabla\tilde{\phi},\Omega\right).
\end{equation*}

Let $\tilde{\phi}_1, \tilde{\phi}_2$ be special invariant Morse functions associated with $\phi$. By Remark \ref{rem:special} we obtain that $\tilde{\phi}_1$ and $\tilde{\phi}_2$ are $\Omega$-homotopic. Therefore Theorem \ref{thm:homotopic} implies that the degree $\degg(\nabla \phi, \Omega)$ is well-defined.

\begin{Remark}\label{rem:propdeg} The degree for $G$-equivariant gradient maps has the properties of excision, additivity, linearisation and homotopy invariance, see \cite{Geba}, \cite{Ryb2005milano}. Moreover, there holds the product formula, see \cite{GolRyb2013}. For more details concerning the theory of equivariant degree, we refer the reader to \cite{BalKra}, \cite{BalKraRyb}.
\end{Remark}

As a corollary of Lemma 3.4 of \cite{GarRyb}, we obtain:
\begin{Fact}\label{fact:lemGarRyb}
Let $\bV, \bV'$ be real, orthogonal $G$-representations and $\nabla_G\text{-}\deg(-Id, B(\bV)) = \nabla_G\text{-}\deg(-Id, B(\bV'))$. Then for some $m,n \in \bN \cup\{0\}$ the representation $\bV\oplus \bR^{2m}$ is  $G$-equivalent to $\bV'\oplus \bR^{2n}$.
\end{Fact}

From the definition of the degree it easily follows:
\begin{Fact}\label{fact:-Idontriv}
If $\bV$ is a trivial $G$-representation then $\nabla_G\text{-}\deg(-Id, B(\bV))=(-1)^{\dim \bV}\cdot\bI$.
\end{Fact}

Now we turn to the definition of the degree for invariant strongly indefinite functionals.
Let $\bH$ be an infinite dimensional Hilbert space, which is an orthogonal $G$-representation. Moreover, assume that there exists a $G$-equivariant approximation scheme on $\bH$, i.e. a sequence of $G$-equivariant projections $\{\pi_n \colon \bH \to \bH,  n \in \bN\cup\{0\}\}$ such that the sequence $\bH^n=\pi_n(\bH)$ of finite dimensional subrepresentations of $\bH$ satisfies $\bH^{n+1}=\bH^n \oplus \bH_{n+1}$ for some subrepresentation $\bH_{n+1} \perp \bH^n$ and $\bH=cl(\bigoplus_{n=1}^{\infty} \bH_n).$
Consider the additional assumptions:
\begin{enumerate}
\item[(d1)] $\Omega \subset \bH$ is an open, bounded, $G$-invariant set,
\item[(d2)] $\Phi \in C^1(\bH, \bR)$  is $G$-invariant and $\Phi(u)=\frac{1}{2} \langle Lu, u\rangle - \eta(u)$, where
\begin{itemize}
\item[(1)] $L \colon \bH \to \bH$ is a linear, bounded, self-adjoint, $G$-equivariant Fredholm operator of index $0$, such that $\bH^0=\ker L$ and $\pi_n \circ L=L \circ \pi_n$ for all $n\in \bN \cup \{0\},$
\item[(2)] $\nabla \eta \colon \bH \to \bH$ is a $G$-equivariant completely continuous operator,
\end{itemize}
\item[(d3)] $(\nabla \Phi)^{-1}(0) \cap \partial \Omega = \emptyset.$
\end{enumerate}

Under the above assumptions, we can define the degree for $G$-invariant strongly indefinite functionals using  the degree for equivariant gradient maps. Namely, we put:
\begin{equation}\label{eq:defdeg}
\degg(\nabla \Phi, \Omega)=\degg(L, B(\bH^{n} \ominus \bH^0))^{-1} \star \degg (L-\pi_{n} \circ \nabla \eta, \Omega \cap \bH^{n})
\end{equation}
for $n$ sufficiently large, see \cite{GolRyb2011} for details. By $\bH^{n} \ominus \bH^0$ we denote the space $\{u\in\bH^n\colon \langle u,v \rangle_{\bH^n}=0 \text{ for all } v\in \bH^0\}$.
Note that we use the same symbol $\degg(\cdot, \cdot)$ for the degree for equivariant gradient maps and the degree for invariant strongly indefinite functionals.

\begin{Remark}
The definition given in \cite{GolRyb2011} is slightly different. Namely, in (d2) there is assumed that $\Phi\in C^1(cl(\Omega), \bR)$. As a consequence, the latter factor in formula \eqref{eq:defdeg} is the degree defined on some restricted set. It is easy to see that using the excision property of the degree for equivariant maps, for $\Phi\in C^1(\bH, \bR)$, one can define the degree by \eqref{eq:defdeg}.
\end{Remark}

\begin{Remark}
The degree defined by \eqref{eq:defdeg} has the properties described in Remark \ref{rem:propdeg}, i.e. the same as the degree for equivariant gradient maps.
\end{Remark}

\subsection{Index of a critical orbit}\label{subsec:indeksorbity}

In this subsection we compute the index of a critical orbit, i.e. the degree at some neighbourhood of this orbit. We consider the finite dimensional case, using the degree for equivariant gradient maps, and the infinite dimensional case with the use of the degree for invariant strongly indefinite functionals.

We start with the finite dimensional case. Let $\bV$ denote a finite dimensional orthogonal $G$-representation and  $\phi\in C^2(\bV,\bR)$ be a $G$-invariant function. Moreover assume that  $G(v_0)\subset(\nabla\phi)^{-1}(0)$ is a non-degenerate critical orbit. By the equivariant Morse lemma, see \cite{Wass}, since $G(v_0)$ is non-degenerate, we can choose a $G$-invariant open set $\Omega\subset\bV$ such that $(\nabla\phi)^{-1}(0)\cap  cl(\Omega)= G(v_0)$.  Under such assumptions $\phi$ is an $\Omega$-Morse function.

Put $\bW=(T_{v_0} G(v_0))^{\bot}$ and $H=G_{v_0}$ and recall that $\bW$ is an orthogonal representation of $H$.  Without loss of generality, we can assume that $\Omega=G\cdot B_{\epsilon}(v_0,\bW)$, where $\epsilon$ is as in Theorem \ref{thm:slice}. Define $\psi\colon B_{\epsilon}(v_0,\bW)\to\bR$ by $\psi=\phi_{|B_{\epsilon}(v_0,\bW)}$.

Note that  from the definitions of $\epsilon$ and $\psi$ we have $(\nabla\psi)^{-1}(0)= H(v_0)=\{v_0\}$. From Lemma \ref{lem:special} there exists $\Omega_0\subset B_{\epsilon}(v_0,\bW)$ and a special $H$-invariant $B_{\epsilon}(v_0,\bW)$-Morse function $\tilde{\psi}\colon B_{\epsilon}(v_0,\bW) \to\bR$ associated with $\psi$. Without loss of generality we can assume that $\Omega_0 =B_{\delta}(v_0,\bW)$ for some $\delta<\epsilon$.

Since $(\nabla\psi)^{-1}(0)\subset (\nabla\tilde{\psi})^{-1}(0)$ and $\tilde{\psi}$ is an $H$-invariant special $B_{\epsilon}(v_0,\bW)$-Morse function, we obtain by the equivariant Morse lemma, see \cite{Wass}, that there exists a finite set $\{v_1,\ldots, v_l\}\subset\bW$ such that
\begin{equation}\label{eq:krytvarphi}
(\nabla\tilde{\psi})^{-1}(0)\cap B_{\epsilon}(v_0,\bW)=\{v_0\}\cup H(v_1)\cup \ldots \cup H(v_l).
\end{equation}

From Lemma \ref{lem:hessian} we have that
\begin{equation*}
\left.\begin{array}{rccc}
 &\bW^H& &\bW^H\\
\nabla^2 \tilde{\psi}(v_0)\colon &\oplus&\to &\oplus\\
&(\bW^H)^{\bot}& &(\bW^H)^{\bot}
\end{array}\right.
\end{equation*}
has the following form
\begin{equation}\label{eq:v0}
\nabla^2 \tilde{\psi}(v_0)=\left[\begin{array}{ccc}
B(v_0)&0\\
0&C(v_0)
\end{array}\right],
\end{equation}
where $m^-(C(v_0))=0$, since $\tilde{\psi}$ is a special $B_{\epsilon}(v_0,\bW)$-Morse function.

For $i=1,\ldots,l$ put $K_i=H_{v_i}$ and $\bU_i=(T_{v_i} H(v_i))^{\bot}\subset\bW$. Using again Lemma \ref{lem:hessian} we obtain that
\begin{equation*}
\left.\begin{array}{rccc}
& T_{v_i} H(v_i)& &T_{v_i} H(v_i) \\
&\oplus& &\oplus\\
\nabla^2 \tilde{\psi}(v_i)\colon &\bU_i^{K_i}&\to&\bU_i^{K_i}\\
&\oplus& &\oplus\\
&(\bU_i^{K_i})^{\bot}& &(\bU_i^{K_i})^{\bot}
\end{array}\right.
\end{equation*}
has the following form
\begin{equation}\label{eq:vi}
\nabla^2 \tilde{\psi}(v_i)=\left[\begin{array}{ccc}
0&0&0\\
0&B(v_i)&0\\
0&0&C(v_i)
\end{array}\right],
\end{equation}
where $m^-(C(v_i))=0$, again because $\tilde{\psi}$ is a special $B_{\epsilon}(v_0,\bW)$-Morse function.

Define $\tilde{\phi} \colon G\cdot B_{\epsilon}(v_0,\bW)\to \bR$ by
\begin{equation}\label{eq:defphi}
\tilde{\phi}(gw) =\tilde{\psi}(w).
\end{equation}
Note that this function is well-defined. Indeed, if $g_1w_1=g_2w_2$, then $w_2 =g_2^{-1}g_1 w_1$ and therefore $g_2^{-1}g_1\in H$ by Theorem \ref{thm:slice2} and hence, by the $H$-invariance of $\tilde{\psi}$,
\[
\tilde{\phi}(g_1w_1)=\tilde{\psi}(w_1)=\tilde{\psi}(g_2^{-1}g_1 w_1)=\tilde{\psi}(w_2)=\tilde{\phi}(g_2w_2).
\]
It is easy to see that $\nabla\tilde{\phi}(gw)=g\nabla \tilde{\psi}(w)$.

Recall that
$
(\nabla\tilde{\psi})^{-1}(0)\cap B_{\epsilon}(v_0,\bW)=\{v_0\}\cup H(v_1)\cup \ldots \cup H(v_l).
$

\begin{Lemma}\label{lem:tildephi0}
The only critical orbits of $\tilde{\phi}$ on $\Omega$ are $G(v_0), G(v_1), \ldots, G(v_l)$. Moreover, $G(v_i)\subset G\cdot B_{\delta}(v_0,\bW)$ for every $i=0,1,\ldots, l$.
\end{Lemma}
\begin{proof}
It is easy to see that $G(v_i)\subset (\nabla\tilde{\phi})^{-1}(0)\cap\Omega$ for $i=0,1,\ldots, l$.

On the other hand, if $v\in (\nabla\tilde{\phi})^{-1}(0)\cap\Omega$, then, by the definition of $\Omega$, there exist $g\in G$ and $w\in B_{\epsilon}(v_0,\bW)$ such that $v=gw$. Therefore
\[
0=\nabla\tilde{\phi}(v)=\nabla\tilde{\phi}(gw)=g\nabla\tilde{\psi}(w)
\]
and hence $w\in(\nabla\tilde{\psi})^{-1}(0)\cap B_{\epsilon}(v_0,\bW) =\{v_0\}\cup H(v_1)\cup \ldots \cup H(v_l)$. Consequently $v\in G(v_i)$ for some $i=0,1,\ldots,l$.

To prove that $(\nabla\tilde{\phi})^{-1}(0)\cap \Omega\subset G\cdot B_{\delta}(v_0,\bW)$ fix $v=gw\in(\nabla\tilde{\phi})^{-1}(0)\cap \Omega$, $w\in B_{\epsilon}(v_0,\bW)$, $g\in G$, and note that,
by Lemma \ref{lem:special}, $w\in B_{\delta}(v_0,\bW)$, i.e. $\|w-v_0\|<\delta$. Therefore $\|gw-gv_0\|=\|w-v_0\|<\delta$, since $\bV$ is an orthogonal representation of $G$.
\end{proof}

Recall that by $K_i$ we denote the isotropy group $H_{v_i}.$

\begin{Lemma}\label{lem:IsGr}
The isotropy group of any element of the critical orbit $G(v_i)$ is conjugated in $G$ to the group $K_i$ for $i=0,1,\ldots,l$.
\end{Lemma}

\begin{proof}
To prove the lemma it is enough to show that  the orbit $G(v_i)$ is $G$-homeomorphic with $G/K_i$, for $i=1,\ldots,l$. This will imply that $(G_{v})_G=(K_i)_G$ for all $v\in G(v_i)$, see Lemma 1.62 of \cite{Kawakubo}.

It is easy to see that $G(v_i) =\theta(G\times_H H(v_i))$, where $\theta$ is the $G$-diffeomorphism given in Theorem \ref{thm:slice}. Since $H(v_i)$ and $H/K_i$ are $H$-homeomorphic (see for example Proposition 1.53 of \cite{Kawakubo}), we obtain that $G\times_H H(v_i)$ is $G$-homeomorphic with $G\times_H (H/K_i)$. To finish the proof note that  from Proposition 1.89 of \cite{Kawakubo} we obtain that $G\times_H (H/K_i)$ is $G$-homeomorphic with $G/K_i$.
\end{proof}

Consider the extension of $\tilde{\phi}$ to the space $\bV$ given by $\tilde{\phi}(v)=\phi(v)$ on $\bV \setminus \Omega$.

\begin{Lemma}\label{lem:tildephi}
The function $\tilde{\phi}$ is a special $G$-invariant $\Omega$-Morse function associated with $\phi$. Moreover, $m^-(\nabla^2 \tilde{\phi}(v_i))=m^-(\nabla^2 \tilde{\psi}(v_i))$ for $i=0,1,\ldots, l$.
\end{Lemma}

\begin{proof}

From Lemma \ref{lem:tildephi0} it follows that $(\nabla\tilde{\phi})^{-1}(0)\cap \Omega=G(v_0)\cup G(v_1)\cup \ldots \cup G(v_l)$. Note that for $i=0,1,\ldots,l$ we have $\bV= T_{v_0}G(v_0)\oplus \bW = T_{v_0}G(v_0)\oplus T_{v_i}H(v_i)\oplus \bU_i$, where $\bU_i=(T_{v_i}H(v_i))^{\bot}\subset\bW$. Moreover $ T_{v_0}G(v_0)\oplus T_{v_i}H(v_i)=T_{v_i}G(v_i)$. Hence, using the formulae \eqref{eq:v0}, \eqref{eq:vi}, we obtain that, for $i=0,1,\ldots,l$,
\begin{equation*}
\left.\begin{array}{rccc}
& T_{v_i} G(v_i)& &T_{v_i} G(v_i) \\
&\oplus& &\oplus\\
\nabla^2 \tilde{\phi}(v_i)\colon &\bU_i^{K_i}&\to&\bU_i^{K_i}\\
&\oplus& &\oplus\\
&(\bU_i^{K_i})^{\bot}& &(\bU_i^{K_i})^{\bot}
\end{array}\right.
\end{equation*}
has the following form
\begin{equation}\label{eq:phivi}
\nabla^2  \tilde{\phi}(v_i)=\left[\begin{array}{ccc}
0&0&0\\
0&B(v_i)&0\\
0&0&C(v_i)
\end{array}\right],
\end{equation}
where $m^-(C(v_i))=0$. This proves that $\tilde{\phi}$ is a special $G$-invariant $\Omega$-Morse function.

To prove that $\tilde{\phi}$ is associated with $\phi$ we need to show that the conditions (1)--(3) of Lemma \ref{lem:special} are satisfied. Note that in Lemma \ref{lem:tildephi0} we have obtained that
$(\nabla\tilde{\phi})^{-1}(0)\cap\Omega\subset \Omega_0$, i.e. we have shown (1).
To prove that $\tilde{\phi}=\phi$ on $\bV\setminus\Omega_0$ we first observe that from the definition of $\tilde{\phi}$ it follows that $\tilde{\phi}=\phi$ on $\bV\setminus\Omega$. Moreover, by the definition of $\tilde{\psi}$, we obtain that $\tilde{\psi}=\psi$ on $B_{\epsilon}(v_0,\bW)\setminus B_{\delta}(v_0,\bW)$ and therefore, for $gw\in \Omega\setminus\Omega_0$,
\[
\tilde{\phi}(gw)=\tilde{\psi}(w)=\psi(w)=\phi(w)=\phi(gw).
\]
This proves (2). Using again the above equality and Lemma \ref{lem:tildephi0} we obtain (3).

To finish the proof note that since $\tilde{\psi}$ is a special $H$-invariant $B_{\epsilon}(v_0,\bW)$-Morse function,  from the formulae \eqref{eq:v0}, \eqref{eq:vi} it follows that $m^-(\nabla^2 \tilde{\psi}(v_i))=m^-(B(v_i))$
for $i=0,1,\ldots, l$. On the other hand, since $\tilde{\phi}$ is a special $G$-invariant $\Omega$-Morse function, we obtain from the formula \eqref{eq:phivi} that $m^-(\nabla^2 \tilde{\phi}(v_i))=m^-(B(v_i))$. Hence $m^-(\nabla^2 \tilde{\phi}(v_i))=m^-(\nabla^2 \tilde{\psi}(v_i))$.
\end{proof}

In the above lemmas we have shown that $\tilde{\phi}$ is a special $G$-invariant $\Omega$-Morse function associated with $\phi$. This allows us to compute the degree $\nabla_G\text{-}\deg(\nabla\phi,\Omega)$. Namely, we have the following:

\begin{Corollary}\label{cor:degree}
From the definition of the degree for $G$-invariant $\Omega$-Morse functions and Lemma \ref{lem:tildephi} we obtain:
\begin{equation*}
\nabla_G\text{-}\deg\left(\nabla\phi,\Omega\right)=\nabla_G\text{-}\deg\left( \nabla\tilde{\phi},\Omega\right).
\end{equation*}
\end{Corollary}

Using Lemmas \ref{lem:tildephi0}--\ref{lem:tildephi} and Corollary \ref{cor:degree} we can obtain formulae for the coordinates of the degree $\nabla_G\text{-}\deg\left(\nabla\phi,\Omega\right)$. We determine them with the use of the restriction $\psi$ of $\phi$ to the space normal to the orbit.

\begin{Theorem}\label{thm:formuladeg} Assume as before that
$(\nabla\tilde{\psi})^{-1}(0)\cap B_{\epsilon}(v_0,\bW)=\{v_0\}\cup H(v_1)\cup \ldots \cup H(v_l).$ Fix $(K)_G\in\sub[G] $ and let $B(v_i)$ be given by \eqref{eq:v0} and \eqref{eq:vi} for $i=0,1, \ldots, l$. Then
\begin{equation*}
\nabla_G\text{-}\deg_{(K)}\left(\nabla\phi,\Omega\right)=\sum\limits_{(H_{v_i})_G=(K)_G} (-1)^{m^-\left(B(v_i)\right)} \in \bZ.
\end{equation*}
\end{Theorem}

\begin{proof}
From Corollary \ref{cor:degree} it follows that we can compute $\nabla_G\text{-}\deg_{(K)}( \nabla\tilde{\phi},\Omega)$ instead of $\nabla_G\text{-}\deg_{(K)}\left(\nabla\phi,\Omega\right)$.
From the formula \eqref{eq:stopienspec} and Lemma \ref{lem:tildephi0} we obtain
\[
\nabla_G\text{-}\deg_{(K)}\left(\nabla\tilde{\phi},\Omega\right)=\sum\limits_{(G_{v_i})_G=(K)_G} (-1)^{m^-\left(\nabla^2\tilde{\phi}(v_i)\right)} \in \bZ.
\]
From Lemma \ref{lem:IsGr} we have $(G_{v_i})_G=(K_i)_G=(H_{v_i})_G$.  Moreover, Lemma \ref{lem:tildephi} and the formula \eqref{eq:phivi} imply that $m^-(\nabla^2\tilde{\phi}(v_i))=m^-(B(v_i))$. Thus
\[
\nabla_G\text{-}\deg_{(K)}\left(\nabla\tilde{\phi},\Omega\right)=\sum\limits_{ (H_{v_i})_G=(K)_G} (-1)^{m^-\left(B(v_i)\right)} \in \bZ.
\]
This completes the proof.
\end{proof}

Now we are going to analyse when the formula for $\nabla_G\text{-}\deg\left( \nabla\phi,\Omega\right)$ can be obtained from the formula for $\nabla_H\text{-}\deg\left(\nabla\psi,B_{\epsilon}(v_0,\bW)\right)$. In general, the coordinates of the $G$-degree of $\phi$ do not have to be in one-to-one correspondence with the coordinates of the $H$-degree of its restriction. We are going to study when there is such a correspondence. To do this we consider the so called admissible pairs of groups, introduced in \cite{PRS}.
Below we recall the notion of such pairs.

\begin{Definition}
Let $H$ be a closed subgroup of $G$. A pair $(G,H)$ is called admissible, if for any closed subgroups $H_1,H_2\subset H$ the following implication is satisfied: if $(H_1)_H\neq (H_2)_H$, then $(H_1)_G\neq (H_2)_G$.
\end{Definition}

The following lemma has been proved in \cite{GolKluSte} (Lemma 2.8):

\begin{Lemma}\label{lem:admissible}
The pair $(\Gamma\times SO(N), \{e\}\times SO(N))$ is admissible.
\end{Lemma}

Denote by $(K_{i_0})_H,(K_{i_1})_H,\ldots,(K_{i_s})_H$ all the possible conjugacy classes of the groups $K_0, K_1, \ldots, K_l$ 
and put
\[
m_j=\nabla_H\text{-}\deg_{(K_{i_j})}\left(\nabla\psi,B_{\epsilon}(v_0,\bW)\right)= \sum\limits_{(K_{i})_H=(K_{i_j})_H} (-1)^{m^-\left(\nabla^2\tilde{\psi}(v_i)\right)} \in \bZ
\]
for $j=0,1,\ldots,s$. Then, from the definition of the degree,
\begin{equation*}
\nabla_H\text{-}\deg\left(\nabla\psi,B_{\epsilon}(v_0,\bW)\right)=\sum\limits_{j=0}^s m_j\cdot \chi_H\left(H/K_{i_j}^+\right)\in U(H).
\end{equation*}

\begin{Theorem}
Under the above notations and assumptions, if the pair $(G,H)$ is admissible,
then
\begin{equation*}
\nabla_G\text{-}\deg\left(\nabla\phi,\Omega\right)=\sum\limits_{j=0}^s m_j\cdot \chi_G\left(G/K_{i_j}^+\right)\in U(G).
\end{equation*}
\end{Theorem}

\begin{proof}

From Lemma \ref{lem:tildephi0} it follows that $(\nabla\tilde{\phi})^{-1}(0)\cap\Omega= G(v_0)\cup G(v_1)\cup \ldots\cup G(v_l)$. From Lemma \ref{lem:IsGr} we obtain that $(G_{v_i})_G=(K_i)_G$ for $i=0,1,\ldots,l$.
Since $(K_{i_0})_H,(K_{i_1})_H,\ldots$, $(K_{i_s})_H$ are all the possible conjugacy classes of $K_0, K_1, \ldots, K_l$ in $H$, $\nabla_G\text{-}\deg_{(\cK)}(\nabla\tilde{\phi},\Omega)=0$ for $(\cK)_G\notin \{(K_{i_0})_G,(K_{i_1})_G,\ldots,(K_{i_s})_G\}$.  Therefore, since $\tilde{\phi}$ is a special $G$-invariant $\Omega$-Morse function associated with $\phi$,
\[
\nabla_G\text{-}\deg\left(\nabla\phi,\Omega\right)= \sum\limits_{j=0}^s\nabla_G\text{-}\deg_{(K_{i_j})}\left(\nabla\tilde{\phi},\Omega\right)\cdot \chi_G\left(G/K_{i_j}^+\right)\in U(G),
\]
where
\[
\nabla_G\text{-}\deg_{(K_{i_j})}\left(\nabla\tilde{\phi},\Omega\right) =\sum\limits_{(G_{v_i})_G=(K_{i_j})_G} (-1)^{m^-\left(\nabla^2\tilde{\phi}(v_i)\right)} \in \bZ.
\]
Moreover, Lemma \ref{lem:tildephi} implies that $m^-(\nabla^2\tilde{\phi}(v_i))=m^-(\nabla^2\tilde{\psi}(v_i))$.
Hence
\[
\nabla_G\text{-}\deg_{(K_{i_j})}\left(\nabla\tilde{\phi},\Omega\right) =\sum\limits_{(K_{i})_G=(K_{i_j})_G} (-1)^{m^-\left(\nabla^2\tilde{\psi}(v_i)\right)}.
\]
To finish the proof note that, by the admissibility of the pair $(G,H)$, we have $(K_{i})_H=(K_{i_j})_H$ if and only if $(K_{i})_G=(K_{i_j})_G$.
\end{proof}

The above theorem allows  to reduce comparing the degrees of critical orbits to comparing the degrees of critical points from the spaces normal to these orbits. More precisely, we have the following:
\begin{Corollary}\label{cor:differentfinitedegrees}
Consider $G$-invariant functions $\phi_i \in C^2(\bV, \bR)$ with non-degenerate critical orbits $G(v_i) \subset (\nabla \phi_i)^{-1}(0)$ for $i=1,2$. Suppose that $G_{v_1}=G_{v_2}= H$ and
put $\bW_i=(T_{v_i} G(v_i))^{\bot}$. Fix $\Omega_i=G\cdot B_{\epsilon}(v_i,\bW_i)$ such that $(\nabla\phi_i)^{-1}(0)\cap \Omega_i= G(v_i)$ and $\epsilon$ satisfies Theorem \ref{thm:slice}. Define $\psi_i\colon B_{\epsilon}(v_i,\bW_i)\to\bR$ by $\psi_i=\phi_{i|B_{\epsilon}(v_i,\bW_i)}$.
If $(G,H)$ is an admissible pair and
\[
\nabla_H\text{-}\deg(\nabla\psi_1,B_{\epsilon}(v_1,\bW_1))\neq \nabla_H\text{-}\deg(\nabla\psi_2,B_{\epsilon}(v_2,\bW_2)),
\]
then
\[
\nabla_G\text{-}\deg(\nabla\phi_1, \Omega_1)\neq \nabla_G\text{-}\deg(\nabla\phi_2, \Omega_2).
\]
\end{Corollary}

Now are going to prove an analogue of Corollary \ref{cor:differentfinitedegrees} in the case of the degree for $G$-invariant strongly indefinite functionals. Fix an infinite dimensional Hilbert space $\bH$, which is an orthogonal $G$-representation.
Suppose that there exists a $G$-equivariant approximation scheme $\{\pi_n \colon \bH \to \bH\colon n \in \bN\cup\{0\}\}$ with $\bH^n=\pi_n(\bH)$.

Let $\Phi\in C^2(\bH,\bR)$ be a $G$-invariant functional and $G(v_0)\subset(\nabla\Phi)^{-1}(0)$ a non-degenerate critical orbit. Fix a $G$-invariant open set $\Omega\subset\bH$ such that $(\nabla\Phi)^{-1}(0)\cap \Omega= G(v_0)$. Put $\bW=(T_{v_0} G(v_0))^{\bot}$,  $H=G_{v_0}$ and $\Psi=\Phi_{|\bW}$. Note that the sequence $\{\tilde{\pi}_n\colon \bW \to \bW \colon n \in \bN \cup \{0\}\}$ of $H$-equivariant orthogonal projections, satisfying $\tilde{\pi}_n(\bW)=\bH^n \cap \bW$, is an $H$-equivariant approximation scheme on $\bW.$ Therefore the degree $\nabla_H\text{-}\deg(\nabla\Psi, \Omega \cap \bW)$ is well-defined.

Consider $G$-invariant functionals $\Phi_i \in C^2(\bH, \bR)$ with non-degenerate critical orbits $G(v_i) \subset (\nabla \Phi_i)^{-1}(0)$ for $i=1,2$. Let $\Phi_i(u) = \frac{1}{2} \langle Lu, u\rangle - \eta_i(u)$, where $L \colon \bH \to \bH$ satisfies condition (d2)(1) of Subsection \ref{subsec:defdegree} and $\nabla\eta_i\colon \bH \to \bH$ are completely continuous $G$-equivariant operators. Suppose that $G_{v_1}=G_{v_2}= H$ and
put $\bW_i=(T_{v_i} G(v_i))^{\bot}$. Define $\Psi_i\colon \bW_i\to\bR$ by $\Psi_i=\Phi_{i|\bW_i}$ and fix $\Omega_i=G\cdot B_{\epsilon}(v_i,\bW_i)$ such that $(\nabla\Phi_i)^{-1}(0)\cap \Omega_i= G(v_i)$, where $\epsilon$ is sufficiently small. 

\begin{Theorem}\label{cor:difdeg}
If $(G,H)$ is an admissible pair and
\[
\nabla_H\text{-}\deg(\nabla\Psi_1,\Omega_1\cap\bW_1)\neq \nabla_H\text{-}\deg(\nabla\Psi_2,\Omega_2\cap\bW_2),
\]
then
\[
\nabla_G\text{-}\deg(\nabla\Phi_1, \Omega_1)\neq \nabla_G\text{-}\deg(\nabla\Phi_2, \Omega_2).
\]
\end{Theorem}

\begin{proof}
For $i=1,2$, from the definition of the degree for strongly indefinite functionals we have
\begin{equation*}
\degg(\nabla \Phi_i, \Omega_i)=\degg(L, B(\bH^n \ominus \bH^0))^{-1} \star \degg (L-\pi_n \circ \nabla \eta_i, \Omega_i \cap \bH^n)
\end{equation*}
for  $n$ sufficiently large.
Suppose that
\[
\nabla_G\text{-}\deg(\nabla\Phi_1, \Omega_1)= \nabla_G\text{-}\deg(\nabla\Phi_2, \Omega_2).
\]
Since $\degg(L, B(\bH^n \ominus \bH^0))^{-1}$ is invertible in $U(G)$,  we obtain
\[
\degg(L-\pi_n \circ \nabla \eta_1, \Omega_1\cap \bH^n)= \degg(L-\pi_n \circ \nabla \eta_2, \Omega_2 \cap \bH^n).
\]
Put $\nabla \psi_i^n=(L-\pi_n \circ \nabla \eta_i)_{|B_{\epsilon}(v_i,\bW_i\cap\bH^n)}$. Without loss of generality we can assume that $\epsilon$ is sufficiently small to satisfy Corollary \ref{cor:differentfinitedegrees}. From this corollary and the above equality  we get
\[
\nabla_H\text{-}\deg(\nabla\psi_1^n,B_{\epsilon}(v_1,\bW_1\cap\bH^n))= \nabla_H\text{-}\deg(\nabla\psi_2^n,B_{\epsilon}(v_2,\bW_2\cap\bH^n)).
\]
Multiplying this equality by $\nabla_G\text{-}\deg(L, B((\bH^n \ominus \bH^0)\cap \bW))^{-1}$, we obtain
\[
\nabla_H\text{-}\deg(\nabla\Psi_1,\Omega_1\cap\bW_1)= \nabla_H\text{-}\deg(\nabla\Psi_2,\Omega_2\cap\bW_2),
\]
which completes the proof.
\end{proof}

\subsection{Global bifurcations from the orbit}\label{subsec:globalbif}
In this subsection we describe the application of the degree to studying the phenomenon of bifurcation from a critical orbit. Let $\bH$ be an infinite dimensional Hilbert space, which is an orthogonal $G$-representation, with a $G$-equivariant approximation scheme $\{\pi_n \colon \bH \to \bH\colon n \in \bN\cup\{0\}\}$. Moreover let $\Phi \in C^2(\bH \times \bR, \bR)$ be a $G$-invariant functional such that $\nabla_u \Phi(u, \lambda)=Lu-\nabla _u \eta(u, \lambda),$ where $L\colon \bH \to \bH$ satisfies condition (d2)(1) of Subsection \ref{subsec:defdegree} and  $\nabla_u \eta \colon \bH \times \bR \to \bH$  is a $G$-equivariant, completely continuous operator.

Assume that  there exists $u_0 \in \bH$ such that $\nabla_u \Phi(u_0, \lambda)=0$ for all $\lambda \in \bR$. Since $\nabla_u \Phi$ is $G$-equivariant, we obtain the critical orbit $G(u_0)$ of $\Phi$ for all $\lambda \in \bR$.  Hence we can consider the family $\cT = G(u_0) \times \bR$ of solutions of the equation $\nabla_u \Phi(u, \lambda)=0.$
We call the elements of $\cT$ the trivial solutions of $\nabla_u\Phi(u, \lambda)=0$. In this subsection we are going to study the existence of nontrivial solutions of this equation, namely the global bifurcation problem.
Since we are interested in the case when critical points are not isolated, we consider $G(u_0) \neq \{u_0\}.$ The case $G(u_0)=\{u_0\}$ has been investigated for example in  \cite{GolRyb2011}.
Put $\cN=\{(u, \lambda) \in \bH \times \bR \colon \nabla_u\Phi(u,\lambda)=0, (u,\lambda)\notin \cT\}.$

\begin{Definition}
A global bifurcation from the orbit $G(u_0) \times \{\lambda_0\} \subset \cT$ of solutions of $\nabla_u \Phi(u, \lambda)=0$ occurs if there is a connected component $\cC(\lambda_0)$ of $cl (\cN)$, containing $G(u_0) \times \{\lambda_0\} $ and such that either $\cC(\lambda_0)\cap (\cT \setminus (G(u_0)\times \{\lambda_0\}))\neq \emptyset$ or $\cC(\lambda_0)$ is unbounded. We call $\lambda_0$ a parameter of a global bifurcation from $\cT$ and denote by $GLOB$ the set of all such parameters.
\end{Definition}

In the following we give the necessary condition for a global bifurcation.

\begin{Fact}\label{fact:podejrzane}
If $\lambda_0 \in GLOB$ then $\dim \ker \nabla^2_u\Phi(u_0, \lambda_0) >\dim(G(u_0) \times \{\lambda_0\}).$
\end{Fact}

The proof of this fact is similar to the proof of Lemma 3.1 of \cite{GolKluSte} and uses the equivariant version of the implicit function theorem of Dancer, see \cite{Dancer}.

Denote by $\Lambda$ the set of $\lambda_0 \in \bR$ such that $\dim \ker \nabla^2_u\Phi(u_0, \lambda_0) >\dim(G(u_0) \times \{\lambda_0\}).$ Fix $\lambda_0 \in \Lambda$ and suppose that there exists $\varepsilon>0$ such that $\Lambda\cap[\lambda_0-\varepsilon,\lambda_0+\varepsilon]=\{\lambda_0\}$.

Since $\lambda_0\pm\varepsilon\notin \Lambda$, by the standard argument we conclude that $G(u_0)\subset \bH$ is an isolated critical orbit of $\Phi(\cdot,\lambda_0\pm \varepsilon)$. Therefore we can choose a $G$-invariant open set $\Omega\subset\bH$ such that $(\nabla_u\Phi(\cdot,\lambda_0\pm\varepsilon))^{-1}(0)\cap \Omega= G(u_0)$.
Note that for $\Omega$ and $\nabla_u \Phi(\cdot,\lambda_0\pm \varepsilon)$ defined as above, the assumptions (d1)--(d3) of Subsection \ref{subsec:defdegree} are satisfied and therefore the degree $\nabla_G\text{-}\deg(\nabla_u\Phi(\cdot,\lambda_0\pm \varepsilon), \Omega)$ is well-defined. We define the bifurcation index $\BIF_G(\lambda_0) \in U(G)$ in the following way:
\begin{equation}\label{eq:defind}
\BIF_G(\lambda_0)=\nabla_G\text{-}\deg(\nabla_u\Phi(\cdot,\lambda_0+\varepsilon), \Omega)-\nabla_G\text{-}\deg(\nabla_u\Phi(\cdot,\lambda_0-\varepsilon), \Omega).
\end{equation}

The nontriviality of the bifurcation index implies the existence of a global bifurcation. Moreover it allows to describe the behaviour of the continuum of solutions, containing $G(u_0) \times \{ \lambda_0\}$. More precisely, there holds the following version of the Rabinowitz alternative:

\begin{Theorem}\label{thm:Rab}
Fix $\lambda_0\in\Lambda$. If $\BIF_G(\lambda_0)\neq \Theta\in U(G)$ then $\lambda_0 \in GLOB.$ Moreover,
either $\cC(\lambda_0)$ is unbounded in $\bH\times\bR$ or
\begin{enumerate}
\item $\cC(\lambda_0)$ is bounded in $\bH\times\bR$,
\item $\cC(\lambda_0)\cap \cT=G(u_0)\times\{\lambda_{i_1},\ldots, \lambda_{i_s}\}$,
\item $\BIF_G(\lambda_{i_1})+\ldots+\BIF_G(\lambda_{i_s})=\Theta\in U(G)$.
\end{enumerate}
\end{Theorem}

The proof of this theorem is standard in the degree theory, see for instance \cite{Ize}, \cite{Nirenberg}, \cite{Rabinowitz} \cite{Rabinowitz1}. In the case of the Leray-Schauder degree it can be found for example in \cite{Brown}. This proof relies on the properties of the degree, namely the homotopy invariance property and the excision property. Moreover, it uses the fact, that the set $(\nabla \Phi)^{-1}(0)\cap \Omega$ is compact. Since in the case of the degree for invariant strongly indefinite functionals such properties remain valid (see \cite{GolRyb2011}), the above version of the Rabinowitz theorem is true also in this situation.

The bifurcation index is defined in terms of the degree in a neighbourhood of the orbit. Therefore, to consider  this index we are going to use results of Subsection \ref{subsec:indeksorbity}. As in this subsection, we put $\bW=(T_{u_0}G(u_0))^{\perp}$, $H=G_{u_0}$ and $\Psi=\Phi_{|\bW}$. Define
$$\BIF_H(\lambda_0)=\degh(\nabla_u \Psi(\cdot, \lambda_0+\varepsilon),\Omega\cap\bW) - \degh(\nabla_u \Psi(\cdot, \lambda_0-\varepsilon),\Omega\cap\bW).$$

It occurs that in some cases instead of verifying the nontriviality of the index $\BIF_G(\lambda_0)\in U(G)$, we can verify the nontriviality of $\BIF_H(\lambda_0)\in U(H).$ More precisely, from Theorems \ref{cor:difdeg} and \ref{thm:Rab}, we have the following:

\begin{Theorem}\label{thm:RabH}
Fix $\lambda_0\in\Lambda$. If the pair $(G,H)$ is admissible and $\BIF_H(\lambda_0)\neq \Theta\in U(H)$ then $\BIF_G(\lambda_0)\neq \Theta\in U(G)$ and consequently the assertion of Theorem \ref{thm:Rab} is satisfied.
\end{Theorem}

\begin{Remark}\label{rem:sum} Using the same reasoning as in the proof of Theorem \ref{cor:difdeg}, it is easy to prove that if the pair $(G,H)$ is admissible and
\[
\BIF_{G}(\lambda_{i_1})+\ldots+\BIF_{G}(\lambda_{i_s})=\Theta\in U(G)
\]
then
\[
\BIF_{H}(\lambda_{i_1})+\ldots+\BIF_{H}(\lambda_{i_s})=\Theta\in U(H).
\]
\end{Remark}

\begin{Remark}
Results concerning the global bifurcation can be obtained also via the Conley index theory. More precisely, using the relationship between the degree and the Conley index (see \cite{Geba}, \cite{GolRyb2013}), one can formulate a global bifurcation theorem in terms of the Conley index, see \cite{GolKluSte}.

Note that the Conley index of a critical orbit can be computed with the use of the index of the critical point of the restricted map. Such a formula has been recently proved in \cite{PRS}.

Although the existence of a global bifurcation of solutions can be in some cases proved by the Conley index theory, using the degree theory approach one can study also the structure of continua of solutions. In particular one can formulate a version of the Rabinowitz alternative, see Theorem \ref{thm:Rab}.
\end{Remark}

\section{Elliptic systems}

Consider the system
\[
\left\{
\begin{array}{rclcl}   A \triangle u & =& \nabla_u F(u,\lambda )   & \text{ in   } & \cU \\
                   \frac{\partial u}{\partial \nu} & =  & 0 & \text{ on    } & \partial\cU,
\end{array}\right.
\]
where $A=\diag\{a_1, a_2,\ldots, a_p\}$, i.e. the system
\begin{equation}\label{eq:system}
\left\{\begin{array}{lclcl}
a_1\Delta u_1&=&\nabla_{u_1}F(u,\lambda) &  \text{in}& \cU \\
a_2\Delta u_2&=&\nabla_{u_2}F(u,\lambda) & \text{in}& \cU\\
&\vdots&&\ \ \ \\
a_p\Delta u_p&=&\nabla_{u_p}F(u,\lambda) & \text{in}& \cU \\
\frac{\partial u_1}{\partial \nu}=\ldots&=&\frac{\partial u_p}{\partial \nu}=0 &\text{on}& \partial\cU,
\end{array}
\right.
\end{equation}
where $\cU\subset\bR^N$ is an open, bounded and $SO(N)$-invariant subset. We assume that
\begin{enumerate}
\item[(a1)] $F\in C^2(\bR^p\times\bR,\bR)$,
\item[(a2)] $u_0$ is a critical point of $F$ for all $\lambda \in \bR$ and there exists a symmetric matrix $B$ such that $\nabla_u^2 F(u_0,\lambda)=\lambda B$.
\item[(a3)] there exist $C>0$ and $1\leq s < (N+2)(N-2)^{-1}$ such that $|\nabla^2_u F(u,\lambda)|\leq C(1+|u|^{s-1})$ (if $N=2$, we assume that $s\in[1,+\infty))$,
\item[(a4)] there exists $0 \leq  p_1\leq p$ such that $a_1=a_2=\ldots=a_{p_1}=-1, a_{p_1+1}=\ldots=a_p=1$, i.e. $A=\diag\{\underbrace{-1,\ldots,-1}_{p_1},\underbrace{1,\ldots,1}_{p_2}\}$, where $p_2=p-p_1,$
\item[(a5)] $B=\left[\begin{array}{cc}
B_1&\bO\\
\bO& B_2
\end{array}\right]$, where $B_1$ and $B_2$ are real symmetric matrices of dimensions $p_1\times p_1$ and $p_2\times p_2$ (respectively) and $\bO$ is the zero matrix of the appropriate dimension,
\item[(a6)]  $\bR^{p_1}$ and $\bR^{p_2}$ are orthogonal $\Gamma$-rep\-re\-sen\-ta\-tions and $F$ is $\Gamma$-invariant, i.e. $F(\gamma u, \lambda) =F(u, \lambda)$ for every $\gamma \in \Gamma, u\in\bR^p, \lambda \in \bR$, for  $\Gamma$ being a compact Lie group.
\end{enumerate}
From (a2) and (a6) it follows that $\nabla_u F(\gamma u_0,\lambda)=0$ for every $\gamma\in\Gamma$, $\lambda\in\bR$, i.e. $\Gamma(u_0)\subset(\nabla_u F(\cdot,\lambda))^{-1}(0)$ for every $\lambda\in\bR$. We additionally assume:
\begin{enumerate}
\item[(a7)] the orbit $\Gamma(u_0)$ is non-degenerate for every $\lambda\in\bR$, i.e.  $\dim \ker \nabla^2 F (u_0,\lambda) = \dim \Gamma (u_0)$.
\end{enumerate}

Recall that if $p_1\cdot p_2>0$, the system is called non-cooperative. We will be mainly interested in such systems.

\subsection{Functional associated with the system.}

Let $H^1(\cU)$ denote the standard Sobolev space with the inner product
\begin{equation*}
\langle \eta,\xi\rangle_{H^1(\cU)}=\int\limits_{\cU}(\nabla \eta(x), \nabla \xi(x)) +\eta(x) \cdot \xi(x)\, dx.
\end{equation*}
Consider the space $\bH=\bigoplus_{i=1}^p H^1(\cU)$ with the inner product given by
\begin{equation}\label{eq:iloczyn}
\langle u, v\rangle_{\bH} = \sum_{i=1}^p \langle u_i,v_i\rangle_{H^1(\cU)}.
\end{equation}
It is known that weak solutions of the system \eqref{eq:system} are in one-to-one correspondence with critical points (with respect to $u$) of the functional $\Phi \colon \bH \times \bR \to \bR$ given by
\begin{equation}\label{eq:Phi}
\Phi(u,\lambda)=\frac{1}{2}\int_{\cU}\sum^p_{i=1}(-a_i|\nabla u_i(x)|^2)\,dx-\int_{\cU}F(u(x),\lambda)\, dx.
\end{equation}

\begin{Remark}
From the assumption (a2) it follows that there exists $g\in C^2(\bR^p\times\bR,\bR)$ such that  $F(u,\lambda) = \frac{\lambda}{2} (Bu, u)-\lambda (B u_0,u) + g(u-u_0,\lambda)$ and for every $\lambda\in\bR$ there hold $\nabla_u g(0,\lambda)=0$ and $\nabla^2_u g(0,\lambda)=0$.
\end{Remark}

Denote by $\tilde{u}_0\in\bH$ the constant function $\tilde{u}_0\equiv u_0$.

\begin{Lemma}\label{lem:postacPhi}
Under the assumptions (a1)--(a4):
\begin{equation*}
\left.\begin{array}{ll}
\nabla_u\Phi(u,\lambda)&=L(u-\tilde{u}_0)+L_{\lambda B} (u-\tilde{u}_0)- \nabla_u\eta_0(u-\tilde{u}_0,\lambda),
\end{array}\right.
\end{equation*}
where
\begin{enumerate}
\item $L\colon\bH\to\bH$ is given by $L(u_1,\ldots,u_p)=(-a_1u_1,\ldots,-a_pu_p)$,
\item $L_{\lambda B}\colon\bH\to\bH$ is given by
$$\langle L_{\lambda B}u,v\rangle_{\bH}=\int_{\cU} (Au(x)-\lambda Bu(x),v(x))\, dx$$
for all $v\in \bH$,
\item $\eta_0\colon \bH\times\bR\to \bR$ is defined by $\displaystyle{\eta_0(u,\lambda) =\int_{\cU} g(u(x),\lambda)\,dx.}$
\end{enumerate}
\end{Lemma}

\begin{proof}
Note that
\begin{eqnarray*}
&&\Phi(u,\lambda)=\frac{1}{2}\int_{\cU}\sum^p_{i=1}\left(-a_i|\nabla u_i(x)|^2\right)\,dx-\int_{\cU}F(u(x),\lambda)\, dx
= \frac{1}{2}\sum^p_{i=1}\left(-a_i\|u_i\|_{H^1(\cU)}^2\right)+\\&& \int_{\cU}\left( \frac12\sum^p_{i=1}a_i|u_i(x)|^2 -\frac{\lambda}{2} (Bu(x), u(x))+\lambda (B u_0,u(x))\right)\,dx
-\int_{\cU} g(u(x)-u_0,\lambda)\, dx =\\
&&\frac12 \langle Lu,u\rangle_{\bH} -\langle L\tilde{u}_0,u\rangle_{\bH}+\\&&\int_{\cU}\left(\frac12(Au(x),u(x)) -(Au_0,u(x)) -\frac{\lambda}{2} (Bu(x), u(x))+\lambda (B u_0,u(x))\right) \,dx - \eta_0(u-\tilde{u}_0,\lambda).
\end{eqnarray*}
Hence, for $v\in \bH$,
\begin{eqnarray*}
\langle\nabla_u \Phi(u,\lambda),v\rangle_{\bH}&=& \langle L(u-\tilde{u}_0),v\rangle_{\bH} +\int_{\cU}(A(u(x)-{u}_0)-\lambda B(u(x)-u_0),v(x)) \, dx-\\ & & \langle\nabla\eta_0(u-\tilde{u}_0,\lambda),v\rangle_{\bH}.
\end{eqnarray*}
This completes the proof.
\end{proof}

Below we collect some properties of $\nabla_u \Phi.$

\begin{Lemma} From the assumptions (a1)--(a4) we obtain:
\begin{enumerate}
\item $L$ is a self-adjoint, bounded Fredholm operator of index 0,
\item $L_{\lambda B}$ is a self-adjoint, bounded, completely continuous operator,
\item $\nabla_u\eta_0\colon\bH\times\bR\to\bH$ is a completely continuous operator such that $\nabla_u\eta_0(0,\lambda)=0,\ \nabla^2_u\eta_0(0,\lambda)=0$ for every $\lambda\in\bR$.
\end{enumerate}
\end{Lemma}

The above properties follow from the definition of the operators $L$ and $L_{\lambda B}$. To prove that the operators are  completely continuous one can use the standard argument, see \cite{Rab}.

Put $\cH_{p_1}=\bigoplus_{i=1}^{p_1} H^1(\cU)$, $\cH_{p_2}=\bigoplus_{i=1}^{p_2} H^1(\cU)$ and note that for $u=(\mathbf{u}_1,\mathbf{u}_2)\in\cH_{p_1}\oplus\cH_{p_2}$ we have $L_{\lambda B}u=(L_{\lambda B_1}\mathbf{u}_1,L_{\lambda B_2}\mathbf{u}_2)$, where $L_{\lambda B_i}\colon\cH_{p_i}\to \cH_{p_i}$, $i=1,2$, and
\[\langle L_{\lambda B_1}\mathbf{u}_1,v\rangle_{\cH_{p_1}}=\int_{\cU} (-\mathbf{u}_1(x)-\lambda B_1\mathbf{u}_1(x),v(x))\, dx \]
for all $v\in \cH_{p_1}$ and
\[\langle L_{\lambda B_2}\mathbf{u}_2,v\rangle_{\cH_{p_2}}=\int_{\cU} (\mathbf{u}_2(x)-\lambda B_2\mathbf{u}_2(x),v(x))\, dx\]
for all $v\in \cH_{p_2}$.

Let us denote  by $\sigma(-\Delta; \cU) = \{ 0=
\alpha_1 < \alpha_2 < \ldots < \alpha_k < \ldots\}$ the set of distinct
eigenvalues of the Laplace operator (with Neumann boundary conditions) on $\cU$. Write $\bV_{-\Delta}(\alpha_k)$ for the eigenspace of $-\Delta$ corresponding to  $\alpha_k \in
\sigma(-\Delta; \cU)$. By the spectral properties of self-adjoint, completely continuous operators, it follows that $H^1(\cU) = cl ( \bigoplus_{k=1}^{\infty} \bV_{-\Delta} (\alpha_k)).$
Let us denote by $\h_k$ the space $\bigoplus_{i=1}^p \bV_{-\Delta}(\alpha_k).$
In particular, for every $u\in \bH$ there exists a unique sequence $\{u^k\}$ such that $u^k\in\h_k$ and $u=\sum_{k=1}^{\infty}u^k$.

Let $ b_1,\ldots, b_{p_1}$ and $ b_{p_1+1},\ldots, b_{p}$ denote the eigenvalues (not necessarily distinct) of $B_1$ and $B_2$ respectively, with corresponding eigenvectors $f_1,\ldots,f_p$ which form an orthonormal basis of $\bR^p$.

Let $\tau_j\colon\bH\to H^1(\cU)$ be a projection such that $\tau_j(u)(x)=(u(x),f_j)$, $j=1,\ldots, p$.
Clearly, if $u^k \in \bH_k,$ then $\tau_j(u^k) \in \bV_{-\Delta}(\alpha_k)$ for $j=1, \ldots, p.$

In the lemma below we characterise the operator $L_{\lambda B}$.

\begin{Lemma}\label{lem:operatorLlambdaB}
For every $u\in\bH$
\[
L_{\lambda B} u=\sum\limits_{k=1}^{\infty}\sum\limits_{j=1}^p \frac{a_j-\lambda b_j}{1+\alpha_k} \tau_j(u^k) \cdot f_j,
\]
where $u=\sum\limits_{k=1}^{\infty}u^k$.
\end{Lemma}
The proof of this lemma is analogous to the proof of Lemma 3.2 of \cite{GolKlu} and it uses the assumption (a5).

In the next sections we will consider eigenspaces of the Laplace operator as $SO(N)$-representations
with the action given by
$
\alpha u(x)=u({\alpha}^{-1}x) $ for $\alpha \in SO(N)$, $u \in \bV_{-\Delta}(\alpha_k)$, $x\in \cU.
$
 In computations we will use the following remark:

\begin{Remark}\label{rem:zero}
It is easy to prove that the eigenspace $\bV_{-\Delta}(0)$ consists only of the constant functions and therefore it is a trivial $SO(N)$-representation, which can be identified with $\bR$. For simplicity, we will write $\bV_{-\Delta}(0)=\bR$.
\end{Remark}

Denote by $\G$ the group $\Gamma \times SO(N)$, where $SO(N)$ is the special orthogonal group in dimension $N$. Note  that the space $\h$ with the scalar product given by \eqref{eq:iloczyn} is an orthogonal ${\G}$-representation with the ${\G}$-action given by
\begin{equation}\label{eq:action}
(\gamma, \alpha) (u)(x)= \gamma u({\alpha}^{-1}x)\  \text{ for }\  (\gamma, \alpha) \in {\G}, u \in \h, x\in \cU.
\end{equation}

As a consequence of (a4)--(a6) we obtain that $L$ is $\cG$-equivariant and $\Phi$ is $\cG$-invariant.

Define a $\cG$-equivariant approximation scheme $ \{\pi_n\colon\h \rightarrow \h\colon n \in \bN \cup \{0\} \}$ on $\h$ by
\begin{enumerate}
\item[(b1)] $\h^0 = \{0\},$
\item[(b2)] $ \h^n =  \bigoplus_{k=1}^n \bigoplus_{j=1}^{p} \bV_{-\Delta} (\alpha_k)$,
\item[(b3)] $\pi_n \colon \h \rightarrow \h$ is a natural $\cG$-equivariant projection such that $ \pi_n(\bH) = \h^n$ for $n \in \bN\cup\{0\}.$
\end{enumerate}

Note that from the definitions of $L$ and $\pi_n$ we have $L \circ \pi_n=\pi_n \circ L$ for all $n \in \bN \cup \{0\}$ and $\ker L = \bH^0$.

Summing up, the results of this section show that for an open, bounded, $\cG$-invariant set $\Omega \subset \bH$ such that $(\nabla \Phi)^{-1} (0) \cap \partial \Omega = \emptyset,$ the conditions (d1)--(d3) of Subsection \ref{subsec:defdegree} are satisfied.

\begin{Corollary}
For an open, bounded, $\cG$-invariant set $\Omega \subset \bH$ such that $(\nabla \Phi)^{-1} (0) \cap \partial \Omega = \emptyset,$ the degree $\nabla_{\cG}\textrm{-}\mathrm{deg}(\nabla \Phi, \Omega)$ is well-defined.
\end{Corollary}

\subsection{Global bifurcations from the orbit}

Recall that by $\tilde{u}_0$ we denote the constant function $\tilde{u}_0\equiv u_0$. From Lemma \ref{lem:postacPhi} it follows that $\tilde{u}_0$ is a critical point of $\Phi$ for every $\lambda \in \bR$. From the $\cG$-invariance of $\Phi$, we obtain that $g\tilde{u}_0$ is also a critical point of this functional for all $g \in \cG$. Consequently we have a critical orbit $\cG(\tilde{u}_0)\subset(\nabla_u \Phi(\cdot,\lambda))^{-1}(0)$ for every $\lambda\in\bR$.
We are going to study a phenomenon of global bifurcation from the set $\cT=\cG(\tilde{u}_0)\times\bR$. We call the elements of $\cT$ the trivial solutions of \eqref{eq:system}. We start with the necessary condition of the bifurcation.

From Fact \ref{fact:podejrzane} we obtain that the global bifurcation can occur only for $\lambda$ such that the orbit $\cG(\tilde{u}_0)\times \{\lambda\}$ is degenerate, i.e. $\dim \ker \nabla^2_u\Phi(\tilde{u}_0, \lambda) >\dim(\cG(\tilde{u}_0) \times \{\lambda\}).$ Denote by $\Lambda$ the set of all such $\lambda$.

\begin{Lemma}\label{lem:Lambdaell} The set $\Lambda$ is given by
\[
\Lambda=\bigcup_{ b_j \in \sigma(B_1) \setminus\{0\}} \bigcup_{\alpha_k \in \sigma(-\Delta;\cU)} \left\{\frac{\alpha_k}{ b_j}\right\}\cup \bigcup_{ b_j \in \sigma(B_2) \setminus\{0\}} \bigcup_{\alpha_k \in \sigma(-\Delta;\cU)} \left\{-\frac{\alpha_k}{ b_j}\right\}.
\]
\end{Lemma}
The proof of this lemma is analogous to the one of Lemma 3.1 of \cite{GolKluSte}.

The consideration of the sufficient condition in general is a difficult problem. To formulate this condition, we consider the additional assumption:

\begin{enumerate}
\item[(a8)] $\Gamma_{u_0}=\{e\}.$
\end{enumerate}
Note that this assumption and formula \eqref{eq:action} imply that $\cG_{\tilde{u}_0}=\{e\}\times SO(N).$
From the Lemma \ref{lem:admissible} it follows that the pair $(\cG, \cG_{\tilde{u}_0})$ is admissible and therefore we can use the results of Subsection \ref{subsec:globalbif}.

Let $\bW\subset\bH$ denote the space normal to the orbit $\cG(\tilde{u}_0)$ at $\tilde{u}_0$. Recall that $\bW$ is a $\cG_{\tilde{u}_0}$-representation (and therefore also an $SO(N)$-representation).

For $\lambda \in \Lambda$ define:
\begin{equation}\label{eq:defV1V2}
\begin{split}
&\bV_1(\lambda)= \bigoplus_{\substack{(\alpha_k, b_j)\in \sigma(-\Delta;\cU)\times \sigma(B_1)\setminus \{0\} \\\lambda  b_j = \alpha_k}}\bV_{-\Delta}(\alpha_k)^{\mu_{B_1}( b_j)},\\
&\bV_2(\lambda)=\bigoplus_{\substack{(\alpha_k, b_j)\in \sigma(-\Delta;\cU)\times \sigma(B_2)\setminus \{0\} \\\lambda  b_j = -\alpha_k}}\bV_{-\Delta}(\alpha_k)^{\mu_{B_2}( b_j)},\\
\end{split}
\end{equation}
where $\mu_{B}(b)$ denotes the multiplicity of $b$ as an eigenvalue of the matrix $B$. Note that we formally understand $\bV_{-\Delta}(\alpha_k)^{\mu_{B}( b)}$ as $\mathrm{span}\{
h\cdot f\colon h\in \bV_{-\Delta}(\alpha_{k}), f\in \bV_{B}(b)\}$, see also \cite{GolKluSte}. However, this space is isomorphic with the direct sum of $\mu_B(b)$ copies of $\bV_{-\Delta}(\alpha_k)$.
Since in our computations the important thing is the dimension of the spaces, we identify these spaces with such direct sums.

\begin{Remark}
From the above definition we can observe that $$\ker \nabla^2_u\Phi(\tilde{u}_0, \lambda)\cap \bW=\bV_1(\lambda)\oplus\bV_2(\lambda).$$
\end{Remark}

\begin{Theorem}\label{thm:glob}
Consider the system \eqref{eq:system} with the potential $F$ and $u_0$ satisfying the assumptions (a1)--(a8). Fix $\lambda_{0} \in \Lambda \setminus \{0\}$ and assume that $\bV_1(\lambda_0)\oplus\bR^{2m}$ is not $SO(N)$-equivalent to $\bV_2(\lambda_0)\oplus\bR^{2n}$ for any $m,n \in \bN\cup\{0\},$ where $\bR^{2m}$, $\bR^{2n}$ are trivial $SO(N)$-representations.
Then a global bifurcation of solutions of \eqref{eq:system} occurs from the orbit $\cG(\tilde{u}_0) \times \{\lambda_0\}.$
\end{Theorem}

\begin{proof}
Choose $\varepsilon>0$ such that $\Lambda\cap[\lambda_{0}-\varepsilon,\lambda_{0}+\varepsilon]=\{\lambda_{0}\}$. From Lemma \ref{lem:Lambdaell} it follows that such a choice is always possible.
Let $\Omega \subset \bH$ be an open, bounded and ${\G}$-invariant subset such that $\nabla_u \Phi (\cdot, \lambda_0\pm\varepsilon)^{-1}(0)\cap  cl(\Omega)={\G}(\tilde{u}_0).$  Without loss of generality, we assume that $ \Omega=\cG\cdot B_{\delta}(\tilde{u}_0,\bW)$, where $\delta$ is sufficiently small.
Then the index $\BIF_{\cG}(\lambda_0)$ given by \eqref{eq:defind} is well-defined. We are going to prove its nontriviality.

Since the pair $(\cG, \cG_{\tilde{u}_0})$ is admissible,
 from Theorem \ref{thm:RabH} it follows that to prove the assertion we have to prove that $\BIF_{\cG_{\tilde{u}_0}}(\lambda_0)\neq\Theta\in U(\cG_{\tilde{u}_0})$. It is easy to see that this condition is equivalent with $\BIF_{SO(N)}(\lambda_0)\neq\Theta\in U(SO(N))$, where $\BIF_{SO(N)}(\lambda_0)$ is defined by:
\begin{equation}\label{eq:BIFSO(N)}
\begin{split}
&\BIF_{SO(N)}(\lambda_0)=\\&\nabla_{SO(N)}\textrm{-}\mathrm{deg}(\nabla_u\Psi(\cdot, \lambda_0+\varepsilon),B_{\delta}(\tilde{u}_0,\bW)) - \nabla_{SO(N)}\textrm{-}\mathrm{deg}(\nabla_u\Psi(\cdot, \lambda_0-\varepsilon),B_{\delta}(\tilde{u}_0,\bW)),
\end{split}\end{equation}
where $\Psi=\Phi_{|\bW}$.

Therefore we are going to concentrate on computing $\nabla_{SO(N)}\textrm{-}\mathrm{deg}(\nabla_u \Psi(\cdot, \hat{\lambda}),B_{\delta}(\tilde{u}_0,\bW))$ for $\hat{\lambda}=\lambda_0\pm\varepsilon$.  Since $\nabla_u^2 \Psi(\tilde{u}_0, \hat{\lambda})$ is an $SO(N)$-equivariant isomorphism, we can apply the linearisation property of the degree. Therefore
 $$\nabla_{SO(N)}\textrm{-}\mathrm{deg}(\nabla_u\Psi(\cdot, \hat{\lambda}),B_{\delta}(\tilde{u}_0,\bW))=\nabla_{SO(N)}\textrm{-}\mathrm{deg}(\nabla_u^2 \Psi(\tilde{u}_0, \hat{\lambda}),B(\bW)).$$

From the definition of the degree for invariant strongly indefinite functionals, since $\bH^0=\{0\}$, we have:
\begin{equation}\label{eq:deglin}
\begin{split}
&\nabla_{SO(N)}\textrm{-}\mathrm{deg}(\nabla_u^2 \Psi(\tilde{u}_0, \hat{\lambda}),B(\bW))=\\
&\nabla_{SO(N)}\textrm{-}\mathrm{deg}(L, B(\bH^n\cap\bW))^{-1}\star \nabla_{SO(N)}\textrm{-}\mathrm{deg}(L-L_{\hat{\lambda} B}, B(\bH^n\cap\bW)),
\end{split}
\end{equation}
where $n$ is sufficiently large.

Now we are going to concentrate on computing the latter factor of \eqref{eq:deglin}.
 Denote by $\cW(\hat{\lambda},n)$ the direct sum of the eigenspaces of $(L-L_{\hat{\lambda} B})_{|\bH^n \cap \bW}$ corresponding to the negative eigenvalues. Since there is no negative eigenvalues of $L-L_{\hat{\lambda} B}$ on $\bH^n \ominus \bW$, we can consider $(L-L_{\hat{\lambda} B})_{|\bH^n }$ instead of $(L-L_{\hat{\lambda} B})_{|\bH^n \cap \bW}$.
 From Lemma \ref{lem:operatorLlambdaB} we have
\[
\sigma(L-L_{\lambda B_1})=\left\{\frac{\alpha_k-\lambda b_j}{1+\alpha_k}\colon  b_j \in \sigma(B_1),\ \alpha_k\in \sigma(-\Delta;\cU)\right\}, \]
\[  \sigma(L-L_{\lambda B_2})=\left\{\frac{-\alpha_k-\lambda b_j}{1+\alpha_k}\colon   b_j \in \sigma(B_2),\ \alpha_k\in \sigma(-\Delta;\cU)\right\}
\]
for $\lambda\in\bR$.
Therefore we obtain:
$$\cW(\hat{\lambda},n)=\h_{(\hat{\lambda}, n, B_1,{+})} \oplus \h_{(\hat{\lambda}, n, B_2,{-})},$$
where, for $\lambda\in\bR$,
\begin{equation}\label{eq:defHbrzydkie}
\begin{split}
&\h_{(\lambda, n, B_1,{+})} = \bigoplus_{k=1}^n \; \bigoplus_{\substack{(\alpha_k, b_j)\in \sigma(-\Delta;\cU)\times \sigma(B_1) \\ \lambda  b_j > \alpha_k}} \bV_{-\Delta}(\alpha_k)^{\mu_{B_1} ( b_j)},\\
&\h_{(\lambda, n, B_2,{-})} = \bigoplus_{k=1}^n \; \bigoplus_{\substack{(\alpha_k, b_j) \in  \sigma(-\Delta;\cU)\times \sigma(B_2) \\ \lambda  b_j > -\alpha_k}} \bV_{-\Delta}(\alpha_k)^{\mu_{B_2} ( b_j)}.
\end{split}
\end{equation}
From the homotopy invariance and the product formula of the degree, we obtain:
\begin{equation}
\begin{split}
&\nabla_{SO(N)}\textrm{-}\mathrm{deg}(L-L_{\hat{\lambda} B}, B(\bW\cap \bH^n))=\\
&  \nabla_{SO(N)}\textrm{-}\mathrm{deg}(-Id, B(\h_{(\hat{\lambda}, n, B_1,{+})}) ) \star \nabla_{SO(N)}\textrm{-}\mathrm{deg}(-Id, B(\h_{(\hat{\lambda}, n, B_2,{-})}) ),
\end{split}
\end{equation}
where in the case when $\bX$ is empty, we put $\nabla_{SO(N)}\textrm{-}\mathrm{deg}(-Id, B(\bX) )=\bI.$
Summing up, we have:
\begin{equation}\label{eq:indlambda}
\begin{split}
&\BIF_{SO(N)}(\lambda_0)= \nabla_{SO(N)}\textrm{-}\mathrm{deg}(L, B(\bH^n\cap\bW))^{-1} \star\\
 & \Big(\nabla_{SO(N)}\textrm{-}\mathrm{deg} (-Id, B(\h_{(\lambda_0+\varepsilon,n, B_1,{+})})) \star  \nabla_{SO(N)}\textrm{-}\mathrm{deg} (-Id, B(\h_{(\lambda_0+\varepsilon,n, B_2,{-})}))-\\&\nabla_{SO(N)}\textrm{-}\mathrm{deg} (-Id, B(\h_{(\lambda_0-\varepsilon,n, B_1,{+})})) \star  \nabla_{SO(N)}\textrm{-}\mathrm{deg} (-Id, B(\h_{(\lambda_0-\varepsilon,n, B_2,{-})}) )\Big).
\end{split}
\end{equation}
Let us first consider the case $\lambda_0>0$.  Since $\lambda_0 \in \Lambda$, we have $\lambda_0=\frac{\alpha_k}{ b_j}$ for $ b_j \in \sigma(B_1)\setminus \{0\}, \alpha_k \in \sigma(-\Delta;\cU)$ or $\lambda_0=-\frac{\alpha_k}{ b_j}$ for $ b_j \in \sigma(B_2)\setminus \{0\}, \alpha_k \in \sigma(-\Delta;\cU)$. Therefore for $ b_j \in \sigma(B_1)$ we have $ b_j>0$ and from the choice of $\varepsilon$ it follows that $(\lambda_0+\varepsilon) b_j>\alpha_k$ iff $\lambda_0  b_j \geq \alpha_k$ and $(\lambda_0-\varepsilon) b_j>\alpha_k$ iff $\lambda_0  b_j >\alpha_k$.
Hence:
\begin{equation*}
\begin{split}
&\bH_{(\lambda_0+\varepsilon, n, B_1, +)}=\bH_{(\lambda_0, n, B_1, +)} \oplus \bV_1(\lambda_0),\\
&\bH_{(\lambda_0-\varepsilon, n, B_1, +)}=\bH_{(\lambda_0, n, B_1, +)}.
\end{split}
\end{equation*}
In the similar way we obtain
\begin{equation*}
\begin{split}
&\bH_{(\lambda_0+\varepsilon, n, B_2, -)}=\bH_{(\lambda_0, n, B_2, -)}, \\
&\bH_{(\lambda_0-\varepsilon, n, B_2, -)}=\bH_{(\lambda_0, n, B_2, -)}\oplus \bV_2(\lambda_0).
\end{split}
\end{equation*}
Therefore
\begin{equation}\label{eq:lambdadodatnie}
\begin{split}
&\BIF_{SO(N)}(\lambda_0)=\nabla_{SO(N)}\textrm{-}\mathrm{deg}(L, B(\bH^n\cap\bW))^{-1} \star\\
&\star \nabla_{SO(N)}\textrm{-}\mathrm{deg} (-Id, B(\h_{(\lambda_0,n, B_1,{+})})) \star  \nabla_{SO(N)}\textrm{-}\mathrm{deg} (-Id, B(\h_{(\lambda_0,n, B_2,{-})}) ))\star\\
&\star\left(\nabla_{SO(N)}\textrm{-}\mathrm{deg} (-Id, B(\bV_1(\lambda_0)))-\nabla_{SO(N)}\textrm{-}\mathrm{deg} (-Id, B(\bV_2(\lambda_0))) \right).
\end{split}
\end{equation}
Analogously, for $\lambda_0<0$ we have
\begin{equation}\label{eq:lambdaujemne}
\begin{split}
&\BIF_{SO(N)}(\lambda_0)=\nabla_{SO(N)}\textrm{-}\mathrm{deg}(L, B(\bH^n\cap\bW))^{-1} \star\\
&\star \nabla_{SO(N)}\textrm{-}\mathrm{deg} (-Id, B(\h_{(\lambda_0,n, B_1,{+})})) \star  \nabla_{SO(N)}\textrm{-}\mathrm{deg} (-Id, B(\h_{(\lambda_0,n, B_2,{-})}) ))\star\\
&\star\left(\nabla_{SO(N)}\textrm{-}\mathrm{deg} (-Id, B(\bV_2(\lambda_0)))-\nabla_{SO(N)}\textrm{-}\mathrm{deg} (-Id, B(\bV_1(\lambda_0))) \right).
\end{split}
\end{equation}

Therefore, since the degrees of $L$ and of $-Id$ are invertible (see \cite{GolRyb2011}), the condition $\BIF_{SO(N)}(\lambda_0)\neq \Theta$ is satisfied if
$$\nabla_{SO(N)}\textrm{-}\mathrm{deg} (-Id, B(\bV_1(\lambda_0))) \neq \nabla_{SO(N)}\textrm{-}\mathrm{deg} (-Id, B(\bV_2(\lambda_0))).$$
Using Fact \ref{fact:lemGarRyb}, we obtain the assertion.

\end{proof}

In the following theorem we consider the case $\lambda_0=0$. Denote by $m^+(B)$ the positive Morse index, i.e. the sum of the multiplicities of the positive eigenvalues of $B$.

\begin{Theorem}\label{thm:globw0} Consider the system \eqref{eq:system} with the potential $F$ and $u_0$ satisfying the assumptions (a1)--(a8). Let $\lambda_0=0$ and assume that $(-1)^{m^+(B)} \neq (-1)^{m^-(B)}$. Then a global bifurcation of solutions of \eqref{eq:system} occurs from the orbit $\cG(\tilde{u}_0) \times \{\lambda_0\}.$
\end{Theorem}

\begin{proof}
Let $\varepsilon>0$ be such that $\Lambda \cap [-\varepsilon, \varepsilon]=\{\lambda_0\}.$
As in the proof of the previous theorem, one can show that the assertion follows from the condition $\BIF_{SO(N)}(0)\neq \Theta \in U(SO(N)).$ Moreover, the formula \eqref{eq:indlambda} remains valid in this case.
In the following we describe spaces $\bH_{( \varepsilon, n, B_1, +)}$,  $\bH_{(- \varepsilon, n, B_2, -)}$.

Considering the condition $\varepsilon  b_j>\alpha_k$ we obtain $ b_j>0$ and hence $\varepsilon>\frac{\alpha_k}{ b_j}$, which, from the choice of $\varepsilon$, is satisfied only for $\alpha_k=0$. Analogously, considering $\varepsilon  b_j>-\alpha_k$ we obtain the following cases:
\begin{enumerate}
\item for $ b_j>0$ this inequality is satisfied for all $\alpha_k \in \sigma(-\Delta; \cU)$,
\item for $ b_j=0$ it is satisfied for $\alpha_k \in \sigma(-\Delta; \cU)\setminus \{0\}$,
\item for $ b_j<0$ this inequality is equivalent to $-\varepsilon>\frac{\alpha_k}{ b_j}$, which again is satisfied for $\alpha_k \in \sigma(-\Delta; \cU)\setminus \{0\}$. Summing up, we have:
\end{enumerate}
\begin{equation*}
\begin{split}
\bH_{(\varepsilon, n, B_1,+)}&=\bigoplus_{ b_j\in\sigma(B_1),  b_j>0} \bV_{-\Delta}(0)^{\mu_{B_1}( b_j)}=\bV_{-\Delta}(0)^{m^+(B_1)},\\
\bH_{(\varepsilon, n, B_2,-)}&=\bigoplus_{k=1}^n \bigoplus_{ b_j\in\sigma(B_2), b_j>0} \bV_{-\Delta}(\alpha_k)^{\mu_{B_2} ( b_j)}
\oplus \bigoplus_{k=2}^n \bigoplus_{ b_j\in\sigma(B_2), b_j\leq0} \bV_{-\Delta}(\alpha_k)^{\mu_{B_2} ( b_j)}=\\&=\bV(n)^{m^+(B_2)}
\oplus  (\bV(n)\ominus\bV(1))^{m^-(B_2)+\mu_{B_2} (0)},
\end{split}
\end{equation*}
where $\bV(n)=\bigoplus_{k=1}^n \bV_{-\Delta} (\alpha_k)$. In the similar way, for $-\varepsilon$ we obtain:
\begin{equation*}
\begin{split}
\bH_{(-\varepsilon, n, B_1,+)}&=\bigoplus_{ b_j\in\sigma(B_1),  b_j<0} \bV_{-\Delta}(0)^{\mu_{B_1}( b_j)}= \bV_{-\Delta}(0)^{m^-(B_1)},\\
\bH_{(-\varepsilon, n, B_2,-)}&=\bigoplus_{k=1}^n \bigoplus_{ b_j\in\sigma(B_2), b_j<0} \bV_{-\Delta}(\alpha_k)^{\mu_{B_2} ( b_j)}\oplus \bigoplus_{k=2}^n \bigoplus_{ b_j\in\sigma(B_2), b_j\geq0} \bV_{-\Delta}(\alpha_k)^{\mu_{B_2} ( b_j)}=\\ &=\bV(n)^{m^-(B_2)} \oplus
(\bV(n)\ominus\bV(1))^{m^+(B_2)+\mu_{B_2} (0)}.
\end{split}
\end{equation*}

Hence, since $m^{\pm}(B_1)+m^{\pm}(B_2)=m^{\pm}(B)$, using Remark \ref{rem:zero} and Fact \ref{fact:-Idontriv} we get
\begin{multline*}
\BIF_{SO(N)}(0)=
\nabla_{SO(N)}\textrm{-}\mathrm{deg}(L, B(\bH^n\cap\bW))^{-1} \star\nabla_{SO(N)}\textrm{-}\mathrm{deg}(-Id, B(\bV(n)\ominus\bV(1)))^{p_2}\star
\\
\star\left(\nabla_{SO(N)}\textrm{-}\mathrm{deg}(-Id, B(\bV_{-\Delta}(0)))^{m^+(B)}- (\nabla_{SO(N)}\textrm{-}\mathrm{deg}(-Id, B(\bV_{-\Delta}(0)))^{m^-(B)}\right)=\\
=\nabla_{SO(N)}\textrm{-}\mathrm{deg}(L, B(\bH^n\cap\bW))^{-1} \star\nabla_{SO(N)}\textrm{-}\mathrm{deg}(-Id, B(\bV(n)\ominus\bV(1)))^{p_2}\star\\
\star\left((-1)^{m^+(B)}- (-1)^{m^-(B)}\right),
\end{multline*}
which is not equal to $\Theta \in U(SO(N))$ if $(-1)^{m^+(B)} \neq (-1)^{m^-(B)}$.
\end{proof}

\begin{Example}
Consider the system \eqref{eq:system} for $\cU$ being the 2-dimensional unit ball $B^2$. Assume that the conditions (a1)--(a8) are satisfied and
\begin{equation}\label{eq:exwarunek1}
\sigma(B_1) \cap \sigma(-B_2) = \emptyset.
\end{equation}
Let $\lambda_0 \in \bR \setminus \{0\}$ be such that
\begin{equation}\label{eq:exwarunek2}
(\sigma(\lambda_0 B_1) \cap (\sigma(-\Delta; B^2)\setminus \{0\}))\cup (\sigma(\lambda_0 B_2) \cap(\sigma(\Delta;B^2) \setminus \{0\}))= \{\alpha\},
\end{equation}
where $\sqrt{\alpha}$ is not a solution of $J_0'(x)=0$, for $J_0$ being the Bessel function of order $0$.

Under such assumptions, $\lambda_0=\frac{\alpha}{b}$  for  $\alpha \in (\sigma(-\Delta; B^2)\cup \sigma(\Delta; B^2))\setminus \{0\} $ and some $b\in \sigma(B) \setminus \{0\}$.
The conditions \eqref{eq:exwarunek1} and \eqref{eq:exwarunek2} imply that
\begin{equation*}
\begin{split}
&\{(\alpha_k,  b_j)\in (\sigma(-\Delta; B^2) \setminus \{0\})\times \sigma(B_1)\colon \lambda_0 b_j=\alpha_k \}\cup\\& \cup \{(\alpha_k,  b_j)\in (\sigma(\Delta; B^2)\setminus \{0\}) \times \sigma(B_2)\colon \lambda_0  b_j=-\alpha_k \}=\{(\alpha, b)\}.
\end{split}
\end{equation*}
Therefore one of the spaces $\bV_1(\lambda_0), \bV_2(\lambda_0)$ is empty and the other one is the direct sum of some number of copies of $\bV_{-\Delta}(\alpha).$ From the description  of the eigenspaces of the Laplace operator (with Neumann boundary conditions) on the ball $B^2$ (see for example Corollary 5.14 of \cite{GolKluSte}), we obtain that $\bV_{-\Delta}(\alpha)$ is a nontrivial $SO(2)$-representation. Therefore the assumptions of Theorem \ref{thm:glob} are satisfied, hence  a global bifurcation of solutions of the system occurs from the orbit $\cG(\tilde{u}_0) \times \{\lambda_0\}.$

Note that to obtain the nontriviality of $\bV_{-\Delta}(\alpha)$ as an $SO(2)$-representation it is enough to verify that $\sqrt{\alpha}$ is a solution of $J_l'(x)=0$ for some $l>0$ and $J_l$ being the Bessel function of order $l$, see Corollary 5.14 of \cite{GolKluSte}.
\end{Example}

\begin{Example}
Replace in the above example $B^2$ by the $N$-dimensional unit ball $B^N$ and suppose that the conditions \eqref{eq:exwarunek1} and \eqref{eq:exwarunek2} are satisfied.
Then if $\sqrt{\alpha}$ is not a root of $J_{\frac{N-2}{2}}'(x)-\frac{N-2}{2x}J_{\frac{N-2}{2}}(x)=0$, for $J_{\frac{N-2}{2}}$ being the Bessel function of order $\frac{N-2}{2}$, then, arguing as in the above example and applying Corollary 5.12 of \cite{GolKluSte}, we again obtain a global bifurcation of solutions of the system from the orbit $\cG(\tilde{u}_0) \times \{\lambda_0\}.$
\end{Example}

\subsection{Unbounded sets of solutions}

In this subsection we are going to study the structure of sets of solutions of the system \eqref{eq:system}, bifurcating from the orbit. To this end we apply the results from Subsection \ref{subsec:globalbif}. We assume that the conditions (a1)--(a8) are satisfied. Moreover, to avoid technicalities, we put an additional assumption, namely
\begin{enumerate}
\item[(a9)] $B_1=\diag \{0,\ldots,0,1,\ldots,1  \}$, $B_2=Id$, in particular $\mu_{B}(0)=\mu_{B_1}(0)=\dim \Gamma(u_0)$.
\end{enumerate}

Obviously, if $p_1-\mu_{B}(0)>0$  then  $\sigma(B_1)=\{0,1\}$ and if $p_2>0$ then $\sigma(B_2)=\{1\}$. From this and Lemma \ref{lem:Lambdaell} it immediately follows:

\begin{Lemma}\label{lem:IlSigma}
Under the above assumptions,
\[\Lambda=\left\{\begin{array}{clc}
\sigma(-\Delta;\cU)\cup \sigma(\Delta;\cU), & \text{ when }& p_1-\mu_{B}(0)>0, p_2>0, \\
\sigma(-\Delta;\cU), & \text{ when }& p_1-\mu_{B}(0)>0, p_2=0, \\
 \sigma(\Delta;\cU), & \text{ when }& p_1-\mu_{B}(0)=0, p_2>0.
\end{array}\right.
\]
\end{Lemma}

Recall that by $\cC(\lambda_0)$ we denote the connected component of the closure of $\cN=\{(u, \lambda) \in \bH \times \bR \colon \nabla_u\Phi(u,\lambda)=0, (u,\lambda)\notin \cG(\tilde{u}_0)\times\bR\}$ containing $\cG(\tilde{u}_0) \times \{\lambda_0\} $.

\begin{Proposition}\label{prop:positive}
Consider the system \eqref{eq:system} with the potential $F$ and $u_0$ satisfying the assumptions (a1)--(a9). Fix $\alpha_{k_0}\in \sigma(-\Delta;\cU)\setminus\{0\}$ such that $\bV_{-\Delta}(\alpha_{k_0})$ is a nontrivial $SO(N)$-representation. If $p_1-\mu_{B}(0)>0$ then $\cC(\alpha_{k_0})\neq\emptyset$.

Moreover, if $p_1-\mu_{B}(0)>0$ is even and the continuum $\cC(\alpha_{k_0})$ is bounded then the following assertions hold:
\begin{enumerate}[(i)]
\item $p_2>0$,
\item $p_2$ is odd,
\item $\cC(\alpha_{k_0})\cap \left(\cG(\tilde{u}_0)\times (\sigma(\Delta,\cU) \setminus \{0\})\right) \neq\emptyset$.
\end{enumerate}
\end{Proposition}

Throughout the proof we keep the notation from the one of Theorem \ref{thm:glob}. Let us start with the following auxiliary lemma. We keep the notation $\bV(n)=\bigoplus_{k=1}^n \bV_{-\Delta} (\alpha_k)$.

\begin{Lemma}\label{lem:doprop}
Fix $\lambda_0\in\Lambda$.
\begin{enumerate}[(i)]
\item If $\lambda_0>0$ then $\lambda_0=\alpha_{k_0}$ for some $\alpha_{k_0}\in \sigma(-\Delta;\cU)\setminus\{0\}$ and
  \begin{multline}\label{eq:bif(+)}
    \BIF_{SO(N)}(\lambda_{0})=\nabla_{SO(N)}\textrm{-}\mathrm{deg}(-Id, B(\bV(k_0-1)))^{p_1-\mu_{B}(0)}\star \\ \left(\nabla_{SO(N)}\textrm{-}\mathrm{deg}(-Id, B(\bV_{-\Delta}(\alpha_{k_0})))^{p_1-\mu_{B}(0)} -\bI\right).
  \end{multline}
\item If $\lambda_0<0$ then $\lambda_0=-\alpha_{k_0}$ for some $\alpha_{k_0}\in \sigma(-\Delta;\cU)\setminus\{0\}$ and
  \begin{multline}\label{eq:bif(-)}
    \BIF_{SO(N)}(\lambda_{0})=\nabla_{SO(N)}\textrm{-}\mathrm{deg}(-Id, B(\bV(k_0)))^{-p_2} \star\\ \left(\nabla_{SO(N)}\textrm{-}\mathrm{deg}(-Id, B(\bV_{-\Delta}(\alpha_{k_0})))^{p_2}-\bI \right).
  \end{multline}
\item  If $\lambda_0=0$ then
\begin{equation}\label{eq:bif(0)}
\BIF_{SO(N)}(\lambda_0)=\left((-1)^{p_1-\mu_{B}(0)}-(-1)^{-p_2}\right)\cdot\bI.
\end{equation}
\end{enumerate}
\end{Lemma}

\begin{proof}
First, recall that from \eqref{eq:lambdadodatnie} and \eqref{eq:lambdaujemne}, we have
\begin{equation*}
\begin{split}
&\BIF_{SO(N)}(\lambda_0)= \sign\hspace{-0.13cm} (\lambda_0)\cdot \nabla_{SO(N)}\textrm{-}\mathrm{deg}(L, B(\bH^n\cap\bW))^{-1}\star \\
& \nabla_{SO(N)}\textrm{-}\mathrm{deg} (-Id, B(\h_{(\lambda_{0},n, B_1,{+})})) \star  \nabla_{SO(N)}\textrm{-}\mathrm{deg} (-Id, B(\h_{(\lambda_{0},n, B_2,{-})}) ))\star\\
&\left(\nabla_{SO(N)}\textrm{-}\mathrm{deg} (-Id, B(\bV_1(\lambda_{0})))-\nabla_{SO(N)}\textrm{-}\mathrm{deg} (-Id, B(\bV_2(\lambda_{0})) \right)
\end{split}
\end{equation*}
for $\lambda_0\neq 0$ and $n$ sufficiently large.

Since $\mu_{B}(0)=\mu_{B_1}(0)$ and the orbit $\Gamma(u_0)$ is non-degenerate,  $T_{u_0}\Gamma(u_0)\subset \bR^{p_1}$. Therefore, using the fact that
$\bH^1=\bV_{-\Delta}(0)^p= \bR^{p_1}\oplus\bR^{p_2}$, we obtain
$\bH^1\cap\bW=(\bR^{p_1}\ominus T_{u_0}\Gamma(u_0))\oplus\bR^{p_2}$.  Consequently
$$\bH^n\cap\bW= \left((\bR^{p_1}\ominus T_{u_0}\Gamma(u_0))\oplus\bigoplus_{k=2}^n \bigoplus_{j=1}^{p_1} \bV_{-\Delta} (\alpha_k)\right)\oplus \left(\bigoplus_{k=1}^n \bigoplus_{j=1}^{p_2} \bV_{-\Delta} (\alpha_k)\right)$$
and $L_{|\bH^n\cap\bW}$ is of the form $(Id, -Id)$ on this decomposition. Hence
\begin{equation}\label{eq:stopienL}
\nabla_{SO(N)}\textrm{-}\mathrm{deg}(L, B(\bH^n\cap\bW))^{-1}=\nabla_{SO(N)}\textrm{-}\mathrm{deg}(-Id, B(\bV(n)))^{-p_2}.
\end{equation}

Consider  $\lambda_0>0$. It follows immediately from Lemma \ref{lem:IlSigma} that $\lambda_0=\alpha_{k_0}$ for some $\alpha_{k_0}\in \sigma(-\Delta;\cU)\setminus\{0\}$ and $p_1-\mu_{B}(0)>0$. We will show \eqref{eq:bif(+)}.

From the formulae \eqref{eq:defV1V2}, \eqref{eq:defHbrzydkie}, we obtain:
\[
\bV_1(\lambda_0)=\bV_{-\Delta}(\alpha_{k_0})^{p_1-\mu_{B}(0)}, \ \ \bV_2(\lambda_0)=\emptyset,
\]
\[\h_{(\lambda_{0},n, B_1,{+})}= \bigoplus_{k=1}^n\bigoplus_{\alpha_k<\alpha_{k_0} } \bV_{-\Delta}(\alpha_{k})^{p_1-\mu_{B}(0)}\ \text{ and }\ \h_{(\lambda_{0},n, B_2,{-})}= \bigoplus_{k=1}^n\bigoplus_{-\alpha_k<\alpha_{k_0}} \bV_{-\Delta}(\alpha_{k})^{p_2}.
\]
 Without loss of generality we can assume that $n>k_0$. Summing up,
\[
\h_{(\lambda_{0},n, B_1,{+})}= \bV(k_0-1)^{p_1-\mu_{B}(0)}\ \text{ and }\ \h_{(\lambda_{0},n, B_2,{-})}=\bV(n)^{p_2}.
\]

Consequently,
\begin{multline*}
\BIF_{SO(N)}(\lambda_{0})=\nabla_{SO(N)}\textrm{-}\mathrm{deg}(-Id, B(\bV(n)))^{-p_2} \star\nabla_{SO(N)}\textrm{-}\mathrm{deg}(-Id, B(\bV(k_0-1)))^{p_1-\mu_{B}(0)}\star\\
\nabla_{SO(N)}\textrm{-}\mathrm{deg}(-Id, B(\bV(n)))^{p_2} \star\left(\nabla_{SO(N)}\textrm{-}\mathrm{deg}(-Id, B(\bV_{-\Delta}(\alpha_{k_0})))^{p_1-\mu_{B}(0)} -\bI\right)=\\
\nabla_{SO(N)}\textrm{-}\mathrm{deg}(-Id, B(\bV(k_0-1)))^{p_1-\mu_{B}(0)} \star\left(\nabla_{SO(N)}\textrm{-}\mathrm{deg}(-Id, B(\bV_{-\Delta}(\alpha_{k_0})))^{p_1-\mu_{B}(0)} -\bI\right),
\end{multline*}
which completes the proof of \eqref{eq:bif(+)}.

If $\lambda_{0}<0$ then clearly $\lambda_{0}=-\alpha_{k_0}$ for some $\alpha_{k_0}\in \sigma(-\Delta;\cU)\setminus\{0\}$ and $p_2>0$. Note that
\[\h_{(\lambda_{0},n, B_1,{+})}= \emptyset, \ \ \h_{(\lambda_{0},n, B_2,{-})}= \bigoplus_{k=1}^n\bigoplus_{\alpha_k>\alpha_{k_0}} \bV_{-\Delta}(\alpha_{k_0})^{p_2}=(\bV(n)\ominus\bV(k_0))^{p_2}
\]
and
\[
\bV_1(\lambda_0)=\emptyset, \ \ \bV_2(\lambda_0)=\bV_{-\Delta}(\alpha_{k_0})^{p_2}.
\]
Therefore
\begin{multline*}
\BIF_{SO(N)}(\lambda_{0})=-\nabla_{SO(N)}\textrm{-}\mathrm{deg}(-Id, B(\bV(n)))^{-p_2} \star\bI\star\\
\nabla_{SO(N)}\textrm{-}\mathrm{deg}(-Id, B(\bV(n)\ominus\bV(k_0)))^{p_2} \star\left(\bI- \nabla_{SO(N)}\textrm{-}\mathrm{deg}(-Id, B(\bV_{-\Delta}(\alpha_{k_0})))^{p_2} \right)=\\
\nabla_{SO(N)}\textrm{-}\mathrm{deg}(-Id, B(\bV(k_0)))^{-p_2}\star\left(\nabla_{SO(N)}\textrm{-}\mathrm{deg}(-Id, B(\bV_{-\Delta}(\alpha_{k_0})))^{p_2}-\bI \right)
\end{multline*}
and this proves \eqref{eq:bif(-)}

Finally, if $\lambda_{0}=0$ then the formula  \eqref{eq:stopienL} remains valid. Hence,  by the proof of Theorem \ref{thm:globw0},
\begin{multline*}
\BIF_{SO(N)}(0)=
\nabla_{SO(N)}\textrm{-}\mathrm{deg}(-Id, B(\bV(n)))^{-p_2}\star\nabla_{SO(N)}\textrm{-}\mathrm{deg}(-Id, B(\bV(n)\ominus\bV(1)))^{p_2}\star\\
\left((-1)^{p_1+p_2-\mu_{B}(0)}- 1\right)= \nabla_{SO(N)}\textrm{-}\mathrm{deg}(-Id, B(\bV_{-\Delta}(0)))^{-p_2}
\star\left((-1)^{p_1+p_2-\mu_{B}(0)}- 1\right).
\end{multline*}
Therefore, since $\bV_{-\Delta}(0)= \bR$  is a trivial $SO(N)$-representation, we obtain, from Fact \ref{fact:-Idontriv}, that
\begin{equation*}
\BIF_{SO(N)}(0)=\left((-1)^{p_1-\mu_{B}(0)}-(-1)^{-p_2}\right)\cdot\bI.
\end{equation*}
This completes the proof.
\end{proof}

\begin{proof}[Proof of Proposition \ref{prop:positive}]

From \eqref{eq:defV1V2} we have
\[
\bV_1(\alpha_{k_0})=\bV_{-\Delta}(\alpha_{k_0})^{p_1-\mu_{B}(0)}, \ \ \bV_2(\alpha_{k_0})=\emptyset.
\]
and therefore from Theorem \ref{thm:glob} we conclude that if $p_1-\mu_{B}(0)>0$ then, since $\bV_{-\Delta}(\alpha_{k_0})$ is a nontrivial $SO(N)$-representation,  $\cC(\alpha_{k_0})\neq \emptyset$.

Suppose that $p_1-\mu_{B}(0)>0$ is even and $\cC(\alpha_{k_0})$ is bounded. By the proof of Theorem \ref{thm:glob}, $\BIF_{\cG}(\alpha_{k_0})\neq\Theta\in U(\cG)$ and therefore, by Theorem \ref{thm:Rab},
$$\cC(\alpha_{k_0})\cap (\cG(\tilde{u}_0)\times\bR)=\cG(\tilde{u}_0)\times\{\lambda_{i_1},\ldots, \lambda_{i_s}\}$$ for some $\lambda_{i_1},\ldots, \lambda_{i_s}\in\Lambda$. Moreover $\alpha_{k_0}\in \{\lambda_{i_1},\ldots, \lambda_{i_s}\}$ and
$$\BIF_{\cG}(\lambda_{i_1})+\ldots+\BIF_{\cG}(\lambda_{i_s})=\Theta\in U(\cG).$$
From the assumption (a8) we have $\cG_{\tilde{u}_0}=\{e\}\times SO(N)$. Therefore from Lemma \ref{lem:admissible} it follows that the pair $(\cG, \cG_{\tilde{u}_0})$ is admissible. Hence, from Remark \ref{rem:sum} we get
\begin{equation}\label{eq:0wSO(N)}
\BIF_{SO(N)}(\lambda_{i_1})+\ldots+\BIF_{SO(N)}(\lambda_{i_s})=\Theta\in U(SO(N)).
\end{equation}
Consider the ring homomorphism $i^*\colon U(SO(N))\to U(SO(2))$ induced by the natural inclusion $i\colon SO(2)\to SO(N)$, see \cite{TomDieck1}. From the properties of $i^*$ it follows that  $i^*(\BIF_{SO(N)}(\lambda_{i_j}))=\BIF_{SO(2)}(\lambda_{i_j})$, where $\BIF_{SO(2)}(\lambda_{i_j})$ is defined analogously as in \eqref{eq:BIFSO(N)}.
Then, from \eqref{eq:0wSO(N)},
\begin{equation}\label{eq:0wSO(2)}
\BIF_{SO(2)}(\lambda_{i_1})+\ldots+\BIF_{SO(2)}(\lambda_{i_s})=\Theta\in U(SO(2)).
\end{equation}
For  $\lambda_{i_j}>0$, we have $\lambda_{i_j}\in \sigma(-\Delta;\cU)$.
Using formula \eqref{eq:bif(+)} and applying the homomorphism $i^*$, we obtain the formula analogous to \eqref{eq:bif(+)} for $\BIF_{SO(2)}(\alpha_{k_0})$. Considering its first factor, using the formula for $\degso(-Id, B(\bX))$ (see for example Theorem 2.3 of \cite{Golebiewska}), we obtain that for even $p_1-\dim \Gamma(u_0)$ the coefficient of $\chi_{\sone}(\sone/\sone^+)$ equals one. In the similar way such coefficient in the latter factor equals zero. Therefore, from the properties of the multiplication in the Euler ring $U(SO(2))$, we obtain that
\begin{equation}\label{eq:positive} \BIF_{SO(2)}(\lambda_{i_j})=\nabla_{SO(2)}\textrm{-}\mathrm{deg}(-Id, B(\bV_{-\Delta}(\lambda_{i_j})))^{p_1-\mu_{B}(0)} -\bI .\end{equation}
Moreover, again from the formula for $\degso(-Id, B(\bX))$, we obtain that all coefficients in the above degree are non-positive.

Following the reasoning from the proof of Lemma 4.3 of \cite{GolRyb2011} we get
\begin{equation}\label{eq:nontrivSO2}
\BIF_{SO(2)}(\alpha_{k_0})\neq \Theta\in U(SO(2))\textbf{}.
\end{equation}

Now we will prove the condition (i) of the assertion. Suppose that $p_2=0$. Then, from Lemma \ref{lem:IlSigma}, $\lambda_{i_j}\geq0$ for every $j=1,\ldots, s$. From \eqref{eq:bif(0)} we obtain $\BIF_{SO(N)}(0)=\Theta\in U(SO(N))$ and consequently $\BIF_{SO(2)}(0)=\Theta\in U(SO(2))$.
Combining it with \eqref{eq:positive} and \eqref{eq:nontrivSO2} we conclude that
\[
\BIF_{SO(2)}(\lambda_{i_1})+\ldots+\BIF_{SO(2)}(\lambda_{i_s})\neq\Theta\in U(SO(2)),
\]
which contradicts \eqref{eq:0wSO(2)}. Hence $p_2>0$.

We turn to the condition (ii) of the assertion. Suppose that $p_2$ is even. Note that from Lemma \ref{lem:doprop}(ii), if $\lambda_{i_j}<0$ then $\lambda_{i_j}\in \sigma(\Delta,\cU)$ and, as before,
\begin{equation*}\label{eq:negative} \BIF_{SO(2)}(\lambda_{i_j})=\nabla_{SO(2)}\textrm{-}\mathrm{deg}(-Id, B(\bV_{-\Delta}(-\lambda_{i_j})))^{p_2} -\bI.
\end{equation*}
As previously, we obtain that all coefficients in the above degree are non-positive.

Moreover, $\BIF_{SO(2)}(0)=\Theta\in U(SO(2))$ by \eqref{eq:bif(0)}. Combining it with \eqref{eq:positive} and \eqref{eq:nontrivSO2} we conclude that
\[
\BIF_{SO(2)}(\lambda_{i_1})+\ldots+\BIF_{SO(2)}(\lambda_{i_s})\neq\Theta\in U(SO(2)),
\]
which contradicts \eqref{eq:0wSO(2)}. Hence $p_2$ is odd.

To finish the proof note that if
\[
\cC(\alpha_{k_0})\cap \left(\cG(\tilde{u}_0)\times (\sigma(\Delta,\cU) \setminus \{0\})\right) =\emptyset,
\]
then
\[
\cC(\alpha_{k_0})\cap\left(\cG(\tilde{u}_0)\times\bR\right)\subset\cC(\alpha_{k_0})\cap \left(\cG(\tilde{u}_0)\times\sigma(-\Delta;\cU)\right)
\]
and therefore applying once again \eqref{eq:bif(0)}, \eqref{eq:0wSO(2)}--\eqref{eq:nontrivSO2} we obtain a contradiction.
\end{proof}

Reasoning similarly as in the above proof we obtain:

\begin{Proposition}\label{prop:negative}
Consider the system \eqref{eq:system} with the potential $F$ and $u_0$ satisfying the assumptions (a1)--(a9). Fix $\alpha_{k_0}\in \sigma(-\Delta;\cU)\setminus\{0\}$ such that $\bV_{-\Delta}(\alpha_{k_0})$ is a nontrivial $SO(N)$-representation. If $p_2>0$ then $\cC(-\alpha_{k_0})\neq\emptyset$.

Moreover, if $p_2>0$ is even and the continuum $\cC(-\alpha_{k_0})$ is bounded then the following assertions hold:
\begin{enumerate}[(i)]
\item $p_1-\mu_{B}(0)>0$,
\item $p_1-\mu_{B}(0)$ is odd,
\item $\cC(-\alpha_{k_0})\cap \left(\cG(\tilde{u}_0)\times (\sigma(-\Delta,\cU) \setminus \{0\})\right) \neq\emptyset$.
\end{enumerate}
\end{Proposition}

As an immediate corollary from Propositions \ref{prop:positive} and \ref{prop:negative} we obtain:

\begin{Theorem}\label{thm:unbounded}
Consider the system \eqref{eq:system} with the potential $F$ and $u_0$ satisfying the assumptions (a1)--(a9). Fix $\alpha_{k_0}\in \sigma(-\Delta;\cU)$ such that $\bV_{-\Delta}(\alpha_{k_0})$ is a nontrivial $SO(N)$-representation. Then
\begin{enumerate}
\item if $p_1-\mu_{B}(0)>0$ and $p_2\geq 0$ are even then the continuum $\cC(\alpha_{k_0})$ is unbounded,
\item if $p_1-\mu_{B}(0)\geq0$ and $p_2> 0$ are even then the continuum $\cC(-\alpha_{k_0})$ is unbounded.
\end{enumerate}
\end{Theorem}



\begin{thebibliography}{99}
\bibitem{BalKra}  Z. Balanov, W. Krawcewicz,  H. Steinlein, \emph{Applied equivariant degree. AIMS Series on Differential Equations \& Dynamical Systems, 1}, American Institute of Mathematical Sciences (AIMS), Springfield, MO, 2006.
\bibitem{BalKraRyb} Z. Balanov, W. Krawcewicz, S. Rybicki, H. Steinlein,  \emph{A short treatise on the equivariant degree theory and its applications}, J. Fixed Point Theory Appl. 8(1) (2010), 1--74.
\bibitem{Bredon} G. E. Bredon, \emph{Introduction to compact transformation groups}, Pure and Applied Mathematics, vol. 46. Academic Press, New York-London, 1972.
\bibitem{Brown} R. F. Brown, \emph{A Topological Introduction to Nonlinear Analysis}, Birkh\"{a}user, Boston, 1993.
\bibitem{Dancer} E. N. Dancer, \emph{The $G$-invariant implicit function theorem\ in infinite dimensions II}, \rm Proceedings of the Royal Society of Edinburgh 102A(3-4) (1986), 211--220.
\bibitem{DuisKolk} J. J. Duistermaat, J. A. C. Kolk, \emph{Lie groups}, Universitext, Springer-Verlag, Berlin, 2000.
\bibitem{GarRyb} G. L. Garza, S. Rybicki, \emph{Equivariant bifurcation index}, Nonlinear Anal. 73(9) (2010), 2779--2791.
\bibitem{Geba} K. G\c{e}ba, \emph{Degree for gradient equivariant maps and equivariant Conley index}, Topological nonlinear analysis II, Birkhäuser (1997),  247--272.
\bibitem{Golebiewska} A. Go{\l}\c{e}biewska, \textit{Periodic solutions of asymptotically linear autonomous Hamiltonian systems}, J. Math. Anal. Appl. 400 (2013) 254--265.
\bibitem{GolKlu} A. Go{\l}\c{e}biewska, J. Kluczenko, \emph{Connected sets of solutions for a nonlinear Neumann problem},  Diff. Int. Equ. 30(11-12) (2017), 833--852.
\bibitem{GolKluSte} A. Go{\l}\c{e}biewska, J. Kluczenko, P. Stefaniak, \emph{Bifurcations from the orbit of solutions of the Neumann problem}, Calc. Var. 57(1) (2018), https://doi.org/10.1007/s00526-017-1285-7.
\bibitem{GolRyb2011} A. Go{\l}\c{e}biewska, S. Rybicki, \emph{Global bifurcations of critical orbits of $G$-invariant strongly indefinite functionals}, Nonlinear Anal. 74(5) \rm (2011), 1823--1834.
\bibitem{GolRyb2013} A. Go{\l}\c{e}biewska, S. Rybicki, \emph{Equivariant Conley index versus the degree for equivariant gradient maps}, Disc. Contin. Dyn. Syst. Ser. S 6(4) (2013), 985--997.
\bibitem{Ize} J. Ize, \emph{Bifurcation theory for Fredholm operators}, Mem. Amer. Math. Soc. 174, 1976.
\bibitem{Kawakubo} K. Kawakubo, \textit{The theory of transformation groups}, Oxford University Press, 1991.
\bibitem{Mayer} K. H. Mayer, \emph{$G$-invariante Morse-functionen}, Man. Math. 63 (1989), 99--114.
\bibitem{Nirenberg} L. Nirenberg, \emph{Topics in Nonlinear Functional Analysis}, Courant Institute of Mathematical Sciences, New York, 1974.
\bibitem{PRS} E. P\'{e}rez-Chavela, S. Rybicki, D. Strzelecki, \emph{Symmetric Liapunov center theorem}, Calc. Var.  56(2) (2017), https://doi:10.1007/s00526-017-1120-1.
\bibitem{Rabinowitz} P. H. Rabinowitz, \emph{Nonlinear Sturm-Liouville problems for second order ordinary differential equations}, Comm. Pure Appl. Math. {23} (1970), 939--961.
\bibitem{Rabinowitz1} P. H. Rabinowitz, \emph{Some global results for nonlinear eigenvalue problems}, J. Functional Analysis  {7} (1971), 487--513.
\bibitem{Rab} P. H. Rabinowitz, {\em Minimax Methods in Critical Point Theory with Applications to
Differential Equations,} CBMS Regional Conference Series in Mathematics, 65, AMS, Providence, R. I., 1986.
\bibitem{Ryb2005milano} S. Rybicki, \emph{Degree for equivariant gradient maps}, Milan J. Math.  73 (2005), 103--144.
\bibitem{RybShiSte} S. Rybicki, N. Shioji, P. Stefaniak, \emph{Rabinowitz alternative for non-cooperative elliptic systems on geodesic balls}, Adv. Nonl. Stud. (2018), https://doi.org/10.1515/ans-2018-0012.
\bibitem{RybSte} S. Rybicki, P. Stefaniak, \emph{Unbounded sets of solutions of non-cooperative elliptic systems on spheres}, J. Differential Equations 259(7) (2015), 2833--2849.
\bibitem{TomDieck1} T. tom Dieck, \emph{Transformation groups and representation theory}, Lecture Notes in Mathematics 766, Springer, Berlin, 1979.
\bibitem{TomDieck} T. tom Dieck, \emph{Transformation groups}, Walter de Gruyter \& Co., Berlin, 1987.
\bibitem{Wass} A. G. Wasserman, \it Equivariant differential topology, \rm Topology  8 \rm (1969), 127--150.
\end{thebibliography}
\end{document}